%% file: genbases_revised_2.tex
\documentclass{amsart}

\usepackage{amssymb}

\usepackage{relsize}
\usepackage[mathscr]{eucal}

\usepackage{graphicx}
\usepackage{comment}

\PassOptionsToPackage{usenames,dvipsnames,svgnames}{xcolor}  
\usepackage{tikz}
\usetikzlibrary{arrows,positioning,automata}
\usetikzlibrary{shapes,snakes}

\theoremstyle{theorem}
\newtheorem{thm}{Theorem}

\theoremstyle{definition}

\newtheorem{lem}[thm]{Lemma}

\newtheorem{claim}[thm]{Claim}

\theoremstyle{definition}
\newtheorem{definition}[thm]{Definition}
\newtheorem{defin}[thm]{Definition}
\newtheorem{remark}[thm]{Remark}

\newcommand{\Q}{\mathbb{Q}}
\newcommand{\R}{\mathbb{R}}
\newcommand{\N}{\mathbb{N}}

\newcommand{\nor}{\mathcal{N}}

\newcommand{\bs}{\boldsymbol{\Sigma}}
\newcommand{\bS}{\boldsymbol{\Sigma}}
\newcommand{\bp}{\boldsymbol{\Pi}}
\newcommand{\bP}{\boldsymbol{\Pi}}
\newcommand{\bd}{\boldsymbol{\Delta}}

\newcommand{\res}{\restriction}

\newcommand{\dt}{D_2(\bP^0_3)}
\newcommand{\ww}{\omega^\omega}

\newcommand{\labeq}[1]{\label{eq:#1}}
\newcommand{\refeq}[1]{(\ref{eq:#1})}
\newcommand{\labt}[1]{\label{thm:#1}}
\newcommand{\reft}[1]{Theorem~\ref{thm:#1}}

\newcommand{\refs}[1]{Section~\ref{section:#1}}
\newcommand{\labs}[1]{\label{section:#1}}
\newcommand{\labf}[1]{\label{fig:#1}}
\newcommand{\reff}[1]{Figure~\ref{fig:#1}}

\newcommand{\dimh}[1]{\hbox{$\dim_{\hbox{H}}$}\left( #1\right)}
\newcommand{\ppq}{\psi_{P,Q}}

\newcommand{\NQ}{\mathscr{N}(Q)}
\newcommand{\NP}{\mathscr{N}(P)}
\newcommand{\Nk}[2]{\mathscr{N}_{#2}( #1 )} 
\newcommand{\NkQ}{\mathscr{N}_k(Q)}

\newcommand{\DNQ}{\mathscr{DN}(Q)}
 
\newcommand{\RNQ}{\mathscr{RN}(Q)}
 
\newcommand{\RNk}[2]{\mathscr{RN}_{#2}( #1 )} 
\newcommand{\RNkQ}{\mathscr{RN}_k(Q)}
\newcommand{\RDN}{\RNQ \cap \DNQ \backslash \NQ}

\newcommand{\tqi}[1]{T_{Q,i}\left( #1 \right)}

\newcommand{\QNK}[2]{Q^{(#1)}_{#2}}
\newcommand{\PxNK}[2]{(P_x)^{(#1)}_{#2}}
\newcommand{\WDN}{\mathscr{WDN}(Q)}

\newcommand{\NN}{\mathbb{N}_2^{\mathbb{N}}}
\newcommand{\IB}{\mathcal{I}_{Q,j}(B)}

\newcommand{\floor}[1]{\left\lfloor #1 \right\rfloor}

\newcommand{\sB}{\mathcal{B}}
\newcommand{\sm}{\setminus}

\newcommand{\vq}{{Q}}
\newcommand{\ri}{\mathscr{R}}

\author[D. Airey]{Dylan Airey}
\address[D. Airey]{Department of Mathematics, Princeton University, Fine Hall, Washington
Road, Princeton, NJ 08544-1000, USA}
\email{dairey@math.princeton.edu}

\author[S. Jackson]{Steve Jackson}
\address[S. Jackson]{Department of Mathematics, University of North Texas, 
General Academics Building 435, 1155 Union Circle,  \#311430, Denton, TX 76203-5017, USA}
\email{stephen.jackson@unt.edu}

\author[B. Mance]{Bill Mance}
\address[B. Mance]{Uniwersytet im. Adama Mickiewicza w Poznaniu,
  Collegium Mathematicum, ul. Umultowska 87, 61-614 Pozna\'{n}, Poland}
\email{william.mance@amu.edu.pl}

\title{Descriptive complexity  in Cantor series}

\begin{document}

\begin{abstract}
A Cantor series expansion for a real number $x$ with respect to a basic
sequence $Q=(q_1,q_2,\dots)$, where $q_i \geq 2$, is a representation of the form
$x=a_0 + \sum_{i=1}^\infty \frac{a_i}{q_1q_2\cdots q_i}$ where $0 \leq a_i<q_i$.
These generalize ordinary base $b$ expansions
where $q_i=b$. Ki and Linton \cite{KiLinton} showed that for ordinary base $b$ expansions the
set of normal numbers is a $\bP^0_3$-complete set, establishing the exact complexity of this set.
In the case of Cantor series there are three natural notions of normality: normality, ratio normality,
and distribution normality (these notions are equivalent for base $b$ expansions). We show that for any
$Q$ the set $\DNQ$ of distribution normal number is $\bP^0_3$-complete, and if $Q$ is
$1$-divergent (i.e., $\sum_{i=1}^\infty \frac{1}{q_i}$ diverges) then the sets $\NQ$ and $\RNQ$
of normal and ratio normal numbers are $\bP^0_3$-complete. We further show that all
five non-trivial differences of these sets are $D_2(\bP^0_3)$-complete if $\lim_i q_i=\infty$
and $Q$ is $1$-divergent (the trivial case is $\NQ\setminus \RNQ=\emptyset$).
This shows that except for the containment
$\NQ\subseteq \RNQ$, these three notions are as independent as possible.
\end{abstract}

\maketitle

\section{Introduction}
\subsection{Review of Definability Notions}
In any topological space $X$, the collection of Borel sets $\sB(X)$ is the smallest 
$\sigma$-algebra containing the open sets. They are stratified into levels,
the Borel hierarchy, by defining $\bs^0_1=$ the open sets, $\bp^0_1=
\neg \bs^0_1= \{ X-A\colon A \in \bs^0_1\}=$the closed sets, 
and for $\alpha<\omega_1$ we let $\bs^0_\alpha$ be the collection of 
countable unions $A=\bigcup_n A_n$ where each $A_n \in \bp^0_{\alpha_n}$
for some $\alpha_n<\alpha$. We also let $\bp^0_\alpha=\neg \bs^0_\alpha$. 
Alternatively, $A\in \bp^0_\alpha$ if $A=\bigcap_n A_n$ where 
$A_n\in \bs^0_{\alpha_n}$ where each $\alpha_n<\alpha$. 
We also set $\bd^0_\alpha= \bs^0_\alpha \cap \bp^0_\alpha$, in particular 
$\bd^0_1$ is the collection of clopen sets. 
For any topological space, $\sB(X)=\bigcup_{\alpha<\omega_1} \bs^0_\alpha=
\bigcup_{\alpha<\omega_1}\bp^0_\alpha$.  All of the collections 
$\bd^0_\alpha$, $\bs^0_\alpha$, $\bp^0_\alpha$ are pointclasses, that is, they are closed 
under inverse images of continuous functions. 
A basic fact (see \cite{Kechris})
is that for any uncountable Polish space $X$, there is no collapse 
in the levels of the Borel hierarchy, that is, all the various 
pointclasses $\bd^0_\alpha$, $\bs^0_\alpha$, $\bp^0_\alpha$, for $\alpha <\omega_1$,
are all distinct. Thus, these levels of the Borel hierarchy can be used
to calibrate the descriptive complexity of a set. We say a set 
$A\subseteq X$ is $\bs^0_\alpha$ (resp.\ $\bp^0_\alpha$) {\em hard} 
if $A \notin \bp^0_\alpha$ (resp.\ $A\notin \bs^0_\alpha$). This says $A$ is 
``no simpler'' than a $\bs^0_\alpha$ set. We say $A$ is $\bs^0_\alpha$-{\em complete}
if $A\in \bs^0_\alpha\sm \bp^0_\alpha$, that is, $A \in \bs^0_\alpha$ and 
$A$ is $\bs^0_\alpha$ hard. This says $A$ is exactly at the complexity level 
$\bs^0_\alpha$. Likewise, $A$ is $\bp^0_\alpha$-complete if $A\in \bp^0_\alpha\sm \bs^0_\alpha$.

A set $D\subseteq X$ is in the class $D_2(\bP^0_3)$ if $D=A\sm B$ where $A, B\in \bP^0_3$.
A set $D$ is $D_2(\bP^0_3)$-hard if $X\sm D \notin D_2(\bP^0_3)$, and
$D$ is $D_2(\bP^0_3)$-complete if it is in $D_2(\bP^0_3)$ and is $D_2(\bP^0_3)$-hard.
As with the classes $\bS^0_\alpha$, $\bP^0_\alpha$, the class $D_2(\bP^0_3)$
has a universal set and so is non-selfdual, that is, it is not closed under
complements (we will define a particular complete set for  $D_2(\bP^0_3)$ in
\S\ref{section:lastcase}).

H.\ Ki and T.\ Linton \cite{KiLinton} proved that the set $\nor(b)$ of
base-$b$ normal numbers (Definition~\ref{normal} below) is
$\bp^0_3(\mathbb{R})$-complete.  Further work was done by V.\ Becher,
P.\ A.\ Heiber, and T.\ A.\ Slaman \cite{BecherHeiberSlamanAbsNormal} who
settled a conjecture of A.\ S.\ Kechris by showing that the set of
absolutely normal numbers is $\bp^0_3(\mathbb{R})$-complete.
Furthermore, V.\ Becher and T.\ A.\ Slaman \cite{BecherSlamanNormal}
proved that the set of numbers normal in at least one base is
$\bs^0_4(\mathbb{R})$-complete. In another direction, D.\ Airey, S.\ Jackson,
D.\ Kwietniak, and B.\ Mance \cite{AJKMSubshifts}, \cite{AJKMDynamics}
showed that for any dynamical system with
a weak form of the specification
property, the set of generic points for the system is $\bP^0_3$-complete. This result
generalizes the Ki-Linton result to many numeration systems other than the standard base $b$ one. 

\subsection{Normal Numbers}
We recall the definition of a normal number.


\begin{defin}\label{normal} 
  A real number $x$ is {\it normal of order $k$ in base $b$} if all blocks of digits of length $k$ in base $b$
  occur with relative frequency $b^{-k}$ in the $b$-ary expansion of $x$. We denote this set by $\Nk{b}{k}$.
  Moreover, $x$ is  {\it normal in base $b$} if it is normal of order $k$ in base $b$ for all natural numbers $k$.
  We denote the set of normal numbers in base $b$ by 
$$
\mathscr{N}(b):=\bigcap_{k \in \mathbb{N}} \Nk{b}{k}.
$$
\end{defin}

We also wish to mention one of the most fundamental and important results relating to normal numbers in base $b$.
The following is due to D. D. Wall in his Ph.D. dissertation \cite{Wall}.

\begin{thm}[D. D. Wall]\labt{wall}
A real number $x$ is normal in base $b$ if and only if the sequence $(b^nx)$ is uniformly distributed mod $1$.
\end{thm}

While it is not difficult to prove \reft{wall}, its importance in the theory of normal numbers can not be understated.
Large portions of the theory of normal numbers in base $b$ make use of \reft{wall}.
We provide an example of a theorem  that provides motivation for the main problem studied in this paper.

\begin{thm}[D. D. Wall]\footnote{A full characterization of $r \in \mathbb{R}$
such that $r+\mathscr{N}(b) \subseteq \mathscr{N}(b)$ is given in \cite{RauzyNormalPreserving}.} \labt{normalpreserves}
  For all rational numbers $q$ (for the second inclusion we assume also $q \neq 0$)
  and integers $b \geq 2$, we have 
\begin{align*}
q+\mathscr{N}(b) &\subseteq \mathscr{N}(b);\\
q\mathscr{N}(b) &\subseteq \mathscr{N}(b).
\end{align*}
That is, normality in base $b$ is preserved by rational addition and multiplication.
\end{thm}

Moreover, \reft{wall} suggests a dynamical interpretation of normality which allows the definition of
normality to be extended to other expansions such as the regular continued fraction expansion,
the L\"{u}roth series expansion, and the $\beta$-expansions.
See \cite{ErgNum} for a basic treatment and introduction to this idea.

In this paper, we are interested in a class of expansions known as the $Q$-Cantor series expansions
that includes the $b$-ary expansions as a special case, but do not admit an extension of \reft{wall}.
The study of normal numbers and other statistical properties of real numbers with respect to
large classes of Cantor series expansions was  first done by P.\ Erd\H{o}s
and A.\ R\'{e}nyi in \cite{ErdosRenyiConvergent} and \cite{ErdosRenyiFurther} and by
A.\ R\'{e}nyi in \cite{RenyiProbability}, \cite{Renyi}, and \cite{RenyiSurvey} and by P.\ Tur\'{a}n in \cite{Turan}.

The $Q$-Cantor series expansions, first studied by G.\ Cantor in \cite{Cantor},
are a natural generalization of the $b$-ary expansions.
G.\ Cantor's motivation to study the Cantor series expansions was to extend the
well-known proof of the irrationality of the number $e=\sum 1/n!$ to a larger class of numbers.
Results along these lines may be found in the monograph of J.\ Galambos \cite{Galambos}. 
Let $\mathbb{N}_k:=\mathbb{Z} \cap [k,\infty)$. If $Q \in \NN$, then we say that $Q$ is a {\it basic sequence}.
Given a basic sequence $Q=(q_i)_{i=1}^{\infty}$, the {\it $Q$-Cantor series expansion}
of a real number $x$  is the (unique)
\footnote{Uniqueness can be proved in the same way as for the $b$-ary expansions.}
expansion of the form
\begin{equation} \labeq{cseries}
x=a_0+\sum_{i=1}^{\infty} \frac {a_i} {q_1 q_2 \cdots q_i}
\end{equation}
where $a_0=\floor{x}$ and $a_i$ is in $\{0,1,\ldots,q_i-1\}$ for $i\geq 1$
with $a_i \neq q_i-1$ infinitely often.
We abbreviate \refeq{cseries} with the notation $x=a_0.a_1a_2a_3\ldots$ w.r.t.\ $Q$.
If $I=[i,j]$ is an interval in $\N$ and the basic sequence $Q$ is understood, we let,
with a slight abuse of notation,
$x \res I$ denote the sequence of digits $a_i,\dots, a_j$. 


For a basic sequence $Q=(q_i)$, a block $B=(e_1, e_2, \cdots, e_k)\in \omega^{<\omega}$,
and a natural number $j$, define
$$
\IB = \begin{cases}
1 &\text{ if } e_1 < q_j, e_2 < q_{j+1}, \cdots, e_k < q_{j+k-1} \\
0 &\text{ otherwise}
\end{cases}
$$
and let
\begin{equation} \label{eqn:Qn}
Q_n(B)=\sum_{j=1}^n \frac {\IB} {q_j q_{j+1} \cdots q_{j+k-1}}.
\end{equation}

Let
$$
Q_n^{(k)}:=\sum_{j=1}^n \frac {1} {q_j q_{j+1} \cdots q_{j+k-1}} \hbox{ and }
T_{Q,n}(x):=\left(\prod_{j=1}^n q_j\right) x \bmod{1}.
$$

$Q_n(B)$ gives the expected number of occurrences of the block $B$ in the Cantor series 
expansion of $x$ 
with a starting position in $[1,n]$. We say $B$ has {\em infinite expectation} if $\lim_{n \to \infty}
Q_n(B)=\infty$. 
$Q_n^{(k)}$ is the expected number of occurrences of $0_k$ (the length $k$ block of $0$s)
with a start in $[1,n]$ (which is also the expected number of occurrences of $1_k$).
We also let $Q_{m,n}(B)=\sum_{j=m}^n \frac {\IB} {q_j q_{j+1} \cdots q_{j+k-1}}$,
which is expected number of occurrences of $B$ with a start in $[m,n]$.

A basic sequence $Q$ is {\it $k$-divergent} if
$\lim_{n \rightarrow \infty} Q_n^{(k)}=\infty$,  {\it fully divergent} if $Q$ is
$k$-divergent for all $k$, and {\it $k$-convergent} if it is not $k$-divergent. 
A basic sequence $Q$ is {\it infinite in limit} if $q_i \rightarrow \infty$.

For a block $B=(e_1,\dots,e_k)$ as above we let $|B|=k$ denote the length of $B$, and let
$\| B\|=(e_1+1)+\cdots+(e_k+1)$. For $1 \leq t \leq k$ we let $B(t)=e_t$ denote the $t$th element
of the block $B$. 

For $x$ a real with $Q$-Cantor series expansion
$a_0.a_1a_2\cdots$, we let $N^Q_n(B,x)$ be the number of $i$
with $1\leq i \leq n$ such that $x\res [i, i+|B|-1]=B$.
We let $N^Q_{m,n}(B,x)$ be the number of $i \in [m, n]$
with $x\res [i,i+|B|-1]=B$. This counts the number of occurrences of the block $B$
with a start in the interval $[m,n]$.

Motivated by \reft{wall}, we make the following definitions of normality for Cantor series expansions.

\begin{definition} A real number $x$  is {\it $Q$-normal of order $k$}
if for all blocks $B$ of length $k$  such that $\lim_{n \to \infty} Q_n(B) = \infty$,
\begin{equation}\labeq{NQdef}
\lim_{n \rightarrow \infty} \frac {N_n^Q (B,x)} {Q_n(B)}=1.
\end{equation}
We let $\Nk{Q}{k}$ be the set of numbers that are $Q$-normal of order $k$. The real number $x$ is {\it $Q$-normal} if
$x \in \NQ := \bigcap_{k=1}^{\infty} \Nk{Q}{k}.$
\end{definition}

\begin{definition}
A real number $x$ is {\it $Q$-ratio normal of order $k$} (here we write $x \in \RNk{Q}{k}$)
if for all blocks $B_1$ and $B_2$ of length $k$  such that 
\\ $\lim_{n \to \infty} \min(Q_n(B_1),Q_n(B_2)) = \infty$, we have
\begin{equation}\labeq{RNQdef}
\lim_{n \to \infty} \frac {N_n^Q (B_1,x)/Q_n(B_1)} {N_n^Q (B_2,x)/Q_n(B_2)}=1.
\end{equation}
We say that $x$ is {\it $Q$-ratio normal} if
$x \in \RNQ := \bigcap_{k=1}^{\infty} \RNk{Q}{k}.$
\end{definition}

\begin{definition}
A real number~$x$ is {\it $Q$-distribution normal} if
the sequence $(T_{Q,n}(x))_{n=0}^\infty$ is uniformly distributed mod $1$.
Let $\DNQ$ be the set of $Q$-distribution normal numbers.
\end{definition}

We note that by \reft{wall}, the analogous versions of the above definitions are equivalent for the $b$-ary expansions.
The situation is far more interesting in the case that $Q$ is infinite in limit and fully divergent.

\begin{figure}
\caption{}
\labf{figure1}
\begin{tikzpicture}[>=stealth',shorten >=1pt,node distance=3.4cm,on grid,initial/.style    ={}]
  \node[state]          (NQ)                        {$\mathsmaller{\NQ}$};
  \node[state]          (RNQ) [above right =of NQ]    {$\mathsmaller{\RNQ}$};
  \node[state]          (RNQDNQ) [below right=of RNQ]    {$\mathsmaller{\RNQ \cap \DNQ}$};
  \node[state]          (NQDNQ) [below right =of NQ]    {$\mathsmaller{\NQ \cap \DNQ}$};
  \node[state]          (DNQ) [above right=of RNQDNQ]    {$\mathsmaller{\DNQ}$};
\tikzset{mystyle/.style={->,double=black}} 
\tikzset{every node/.style={fill=white}} 
\path (RNQDNQ)     edge [mystyle]    (RNQ)
      (RNQDNQ)     edge [mystyle]     (DNQ)
      (NQ)     edge [mystyle]     (RNQ)
      (NQDNQ)     edge [mystyle]     (RNQDNQ)
      (NQDNQ)     edge [mystyle]     (NQ);
\tikzset{mystyle/.style={<->,double=black}}
\end{tikzpicture}
\end{figure}

It was proved in \cite{ppq1} that the directed graph in \reff{figure1} gives the
complete containment relationships between these notions when $Q$ is infinite in limit and fully divergent.
The vertices are labeled with all possible intersections of one, two, or three choices
of the sets $\NQ$, $\RNQ$, and $\DNQ$, where we know that $\NQ=\NQ \cap \RNQ$ and $\NQ \cap \DNQ=\NQ \cap \DNQ \cap \RNQ$.
The set labeled on vertex $A$ is a subset of the set labeled on vertex $B$ if and only if
there is a directed path from $A$ to $B$.  For example, $\NQ \cap \DNQ \subseteq \RNQ$, so all
numbers that are $Q$-normal and $Q$-distribution normal are also $Q$-ratio normal. 

We remark that all inclusions suggested from \reff{figure1} are either
easily proved ($\NQ \subseteq \RNQ$) or are trivial.
The difficulty comes in showing a lack of inclusion.
The most challenging of these is to prove that there is a basic sequence $Q$ where $\RDN\neq \emptyset$.

As the equivalence of these definitions is so key to the study of
normality in base $b$, it is natural to ask how ``independent'' these
sets are.  There have been two approaches to measure this.  First, it
is natural to ask if, for example, there is a simple condition $P(x)$
where if $x$ is $Q$-normal and $P(x)$ also holds, we will have the $x$
is $Q$-distribution normal (or any other permutation of definitions of
normality).  One example of such an attempt to find a condition $P(x)$
is motivated by \reft{normalpreserves}.  This theorem strongly fails
when $Q$ is infinite in limit and fully divergent: $Q$-distribution
normality is preserved only by non-zero integer multiplication while
$Q$-normality and $Q$-ratio normality aren't even preserved by integer
multiplication.  In fact, the easiest ways to construct members of
$\NQ \backslash \DNQ$ is to use the techniques presented in
\cite{ppq1} to construct members of the (surprisingly) non-empty set
$$ \{x \in \NQ : nx \notin \NQ \forall n \in \mathbb{N}_2\}.
$$ This motivated Samuel Roth to ask the third author if it is true
that $nx \in \NQ$ for all natural numbers $n$ implies that $x \in
\DNQ$ at the 2012 RTG conference: Logic, Dynamics and Their
Interactions, with a Celebration of the Work of Dan Mauldin in Denton,
Texas.  This question was later strongly shown to be false in
\cite{AireyManceVandehey} as it was shown that there exist basic
sequences $Q$ such that
$$ \dimh{ \{ x \in \mathbb{R} : rx+s \in \NQ \backslash \DNQ \forall r
\in \Q \backslash \{0\}, s \in \Q \}}=1.
$$ Any other attempt to find such an additional condition that would
allow one to get from one form of normality to another has thus far
failed.

The second method has been to attempt to find the ``size'' of the
difference sets suggested by \reff{figure1}.  All of these difference
sets are meager and have zero measure.  The Hausdorff dimension of
most of these difference sets has been calculated in
\cite{AireyManceHDDifference}.  In particular, when $Q$ is infinite in
limit and fully divergent, all non-empty difference sets except for
$\NQ \backslash \DNQ$ are known to have full Hausdorff dimension.  The
Hausdorff dimension of $\NQ \backslash \DNQ$ is known to have full
Hausdorff dimension only for a small class of infinite in limit and fully
divergent basic sequences $Q$.

Another approach to measuring the difference sets is to determine the exact descriptive complexity of
these sets. If we show that for two of these classes (for a given $Q$) the difference set is
$D_2(\bP^0_3)$-complete, then the difference has maximum logical complexity.
One of the  main results of this paper, Theorem~\ref{differences},  is to establish this fact for the five non-trivial
difference sets. As we mention in some examples below,
this can be used to rule out potential theorems connecting the different classes.

\begin{remark} \label{equivcond}
When $Q$ is infinite in limit and $k$-divergent,
conditions \refeq{NQdef} and \refeq{RNQdef} can be replaced by
$$ \lim_{n \rightarrow \infty} \frac {N_n^Q (B,x)} {Q_n^{(k)}}=1
\hbox{ and } \lim_{n \to \infty} \frac {N_n^Q (B_1,x)} {N_n^Q
(B_2,x)}=1,
$$ respectively.  This class of expansions will be important for us
throughout this paper.  Moreover, $\Nk{Q}{k}=\RNk{Q}{k}=\mathbb{R}$ if
and only if $Q$ is $k$-convergent.

\end{remark}

\begin{remark} \label{rdn}
Assuming $Q=(q_i)$ is infinite in limit, we have the following easy facts about distribution normality.
A  real $x$ with $Q$-Cantor series expansion
$a_0.a_1a_2\cdots$ is in $\DNQ$ iff the sequence $(\frac{a_i}{q_i})$
is uniformly distributed mod 1. If $y=b_0.b_1b_2\cdots$ and $\lim_{i \to \infty} \frac{a_i-b_i}{q_i}
=0$, then $x \in \DNQ$ iff $y \in \DNQ$. Also, if $\{ i \colon a_i \neq b_i\}$
has density $0$, then $x \in \DNQ$ iff $y \in \DNQ$.
\end{remark}


We will need the following theorem of \cite{Mance4}.
\begin{thm}\labt{measure}
The sets $\NQ, \RNQ$, and $\DNQ$ are sets of full measure for all basic sequences $Q$.\footnote{  The definitions of $Q$-normality and $Q$-ratio normality that were used were less general, but the general result holds with only small modification for our more general definition.}
\end{thm}

\subsection{Statement of Results}

We will prove the following theorems. First we address the complexity of the
various normality classes themselves. Theorems~\ref{thm:DNQ} and \ref{NandRN}
can be seen as generalizations of the Ki-Linton result as $\NQ$, $\RNQ$, and $\DNQ$
all coincide when $Q$ is the constant $b$ sequence. The proofs, particularly that of
Theorem~\ref{NandRN}, have however significant extra complications.

\begin{thm} \label{thm:DNQ}
For all basic sequences $Q$, the set $\DNQ$ is $\bp^0_3$-complete.
\end{thm}

\begin{thm} \label{NandRN}
The sets $\NQ$ and $\RNQ$ are $\bp^0_3$-complete if $Q$ is $1$-divergent, and clopen if $Q$ is $1$-convergent.
Moreover, $\NkQ$ and $\RNkQ$ are $\bp^0_3$-complete if $Q$ is $k$-divergent and clopen if $Q$ is $k$-convergent.
\end{thm}

We can extend Theorem~\ref{NandRN} to show the following.

\begin{thm} \label{genblocks}
Let $\mathscr{C}$ be a collection of blocks.  Then the set
$$
\mathscr{N}_\mathscr{C}(Q)=\left\{x \in \mathbb{R} \colon \lim_{n \to \infty} \frac {N_n^Q(B,x)} {Q_n(B)} =1\
\forall B\in \mathscr{C} \hbox{ such that } \lim_{n \to \infty} Q_n(B) = \infty \right\}
$$
is $\bp^0_3$-complete if there exists $B \in \mathscr{C}$ such that
$\lim_{n \to \infty} Q_n(B) = \infty$ and clopen otherwise.  Similarly, the set

\begin{equation*}
\begin{split}
\mathscr{RN}_\mathscr{C}(Q)=& \left\{ x \in \mathbb{R}  \colon \lim_{n \to \infty}
\frac {N_n^Q(B_1,x)/Q_n(B_1)} {N_n^Q(B_2,x)/Q_n(B_2)}=1 \ \forall B_1,B_2\in \mathscr{C} \hbox{ such that } \right.
\\ & \qquad 
 |B_1|=|B_2| \hbox{ and }  \left. \lim_{n \to \infty} \min(Q_n(B_1),Q_n(B_2)) =
\vphantom{\frac {N_n^Q(B_1,x)/Q_n(B_1)} {N_n^Q(B_2,x)/Q_n(B_2)}}
\infty \right\}
\end{split}
\end{equation*}
is $\bp^0_3$-complete if there exist $B_1, B_2 \in \mathscr{C}$ such that
$\lim_{n \to \infty} \min(Q_n(B_1),Q_n(B_2)) = \infty$
and $\mathscr{C}$ satisfies the following hypothesis:

$(\star)$ For every $B\in \mathscr{C}$ with infinite expectation there is a block
$B' \in \mathscr{C}$ of infinite expectation with $|B'|=|B|$ and an integer $1 \leq t \leq |B|$
such that $|B(t)-B'(t')|>1$ for all $1 \leq t' \leq |B'|$.

If there do not exist $B_1, B_2 \in \mathscr{C}$ such that
$\lim_{n \to \infty} \min(Q_n(B_1),Q_n(B_2)) = \infty$,
then $\mathscr{RN}_\mathscr{C}$ is clopen.
\end{thm}

\begin{remark}
The $\bp^0_3$-completeness of $\mathscr{N}_{\mathscr{C}}$
holds for general $\mathscr{C}$ (and all $Q$),
but the $\bp^0_3$-completeness of $\mathscr{RN}_{\mathscr{C}}$ requires the
extra hypothesis $(\star)$ on $\mathscr{C}$. We do not know if this extra hypothesis is necessary.
For base $b$ expansions, normality and ratio normality coincide so the extra assumption
is not needed, but we do not know for general $Q$.
\end{remark}

\begin{remark}
The proof of Theorem~\ref{genblocks} will also show the $\bp^0_3$-completeness for
a variation of ratio normality which we call {\em strong ratio normality}. Here
we remove the restriction that $|B_1|=|B_2|$ in the above definition of
$\mathscr{RN}_\mathscr{C}(Q)$. We can accordingly relax the $(\star)$ assumption
by removing the requirement that $|B'|=|B|$.
\end{remark}

The next theorem addresses the complexity of the difference sets. We note that the hypotheses on $Q$
are necessary as the various normality classes coincide for base $b$ expansions (where $Q$
is not infinite in limit) and when $Q$ is not $1$-divergent then $\NQ$ and $\RNQ$ are clopen.

\begin{thm}\label{differences}
Assume that $Q$ is infinite in limit  and $1$-divergent. Then the sets
$\DNQ\sm \NQ$, $\DNQ\sm \RNQ$, $\NQ\sm \DNQ$,  $\RNQ\sm \DNQ$, and $\RNQ\sm \NQ$ are all $\dt$-complete.
\end{thm}







Theorem~\ref{differences} imposes limitations on the relationships between
the classes $\NQ$, $\RNQ$, and $\DNQ$. For example, consider the sets $\NQ$ and $\RNQ$.
Since $\RNQ \sm \NQ$ is $D_2(\bP^0_3)$-complete, there cannot be a $\bS^0_3$ set
$A$ such that $A\cap \RNQ=\NQ$ (as otherwise we would have $\RNQ\sm \NQ=\RNQ\sm A \in \bP^0_3$,
a contradiction). Thus, no $\bS^0_3$ condition can be added to the assumption of
ratio normality to give the set of normal numbers.
Equivalently, anytime a $\bS^0_3$ set contains $\NQ$ (or $\DNQ$), then it must
contain elements of $\RNQ\sm \NQ$ and $\DNQ \sm \NQ$ (resp.\ $\NQ\sm \DNQ$, and so $\RNQ\sm \DNQ$). 
For example, though $\NQ$ has Lebesgue measure
one, any $\bP^0_2$ set of measure one which contains $\NQ$ must contain an element
of $\RNQ\sm \NQ$, as well as $\DNQ\sm \NQ$. Many naturally occurring sets of reals $A$ are defined by conditions
which result in them being $\bS^0_3$ sets.
Examples include countable sets, co-countable sets, the class BA of {\em badly approximable}
numbers (which is a $\bS^0_2$ set), the Liouville numbers (which is a $\bP^0_2$ set),
and the set of $x\in [0,1]$ where a particular continuous function $f\colon [0,1]\to \R$
is not differentiable. 
In all these cases
the theorem implies that either the set omits some normal number, or else contains
a number which is ratio normal but not normal (and likewise for $\DNQ$). Of course, many of these statements
are easy to see directly, but the point is that they all follow immediately from the general
complexity result, Theorem~\ref{differences}.

Previous work of Mance \cite{Mance} had shown that all of the non-trivial difference sets
(all except $\NQ\sm \RNQ$, which is trivially empty) are non-empty assuming $Q$ is infinite in limit
$k$-divergent for all $k$. Thus, Theorem~\ref{differences} strengthens this
in two ways: we relax the hypothesis to $Q$ being $1$-divergent, and we show the
difference sets are actually $D_2(\bP^0_3)$-complete.

To mention another  application of Theorem~\ref{differences}, consider
(relative to a fixed basic sequence $Q$ which is infinite in limit and $1$-divergent)
the following weakening of distribution normality. Say a real $x$ is $\epsilon$-{\em weakly distribution normal}
(for $\epsilon >0$) if there is an $\ell$ such that for all $N\geq \ell$ and all
intervals $(a,b) \subseteq (0,1)$

\[
\left| \frac{1}{N} \# \{ n <N\colon q_0 \cdots q_{n-1} x \mod 1 \in [a,b]\} -(b-a)\right| \leq \epsilon.
\]
The set $\WDN_\epsilon$ of $\epsilon$-weakly distributional normal numbers is easily a $\bS^0_2$ set.
It therefore follows from Theorem~\ref{differences} that
$\NQ\cap \WDN_\epsilon \sm \DNQ$ is non-empty.

For one more example, let $\ri(Q)$ denote the set of {\em rich} numbers. These are the $x$ such that
every block $B \in \omega^{<\omega}$ occurs in the $Q$-ary expansion of $x$. The set $\ri(Q)$
is easily a $\bP^0_2$ set, and contains the ratio normal numbers (and so also the normal numbers).
From Theorem~\ref{differences} it therefore follows that $\DNQ\cap \ri(Q) \sm \NQ$ is
non-empty, 
in other words, distribution normal and rich does not imply normal (if this failed, then
$\DNQ\sm \NQ$ would be equal to $\DNQ\sm \ri(Q)$ which a $\bP^0_3$ set, contradicting
Theorem~\ref{differences}).

\section{$\bP^0_3$-completeness of the normality classes}

Throughout, $\vq=(q_i)_{i=1}^\infty$ will denote a {\em basic sequence}, that is, a
sequence of integers $q_i$ with $q_i \geq 2$
for all $i$.

Note that the $\bp^0_3$-completeness of the set $\NQ \cap [0,1]$ (and likewise for
$\RNQ\cap [0,1]$ and $\DNQ \cap [0,1]$) immediately implies the
$\bp^0_3$-completeness of $\NQ$, since if $\NQ$ were in $\bs^0_3$, then so would be
$\NQ\cap [0,1]$. 
Similarly, the $\bp^0_3$-completeness of the difference sets
restricted to $[0,1]$ (for example $\RNQ \sm \NQ \cap [0,1]$)
implies the $\bp^0_3$-completeness of the difference set (e.g.\ $\RNQ\sm \NQ$).
So, for the rest of the paper we will restrict our attention to reals in $[0,1]$, that is, 
when we write $\NQ$ etc., we will henceforth mean $\NQ\cap [0,1]$.

The basic sequence $Q$ forms the set of bases for the expansion of an real $x \in [0,1]$
into a Cantor series 
\[
x= \sum_{i=1}^\infty \frac{a_i}{q_1 \cdots q_i},
\]
where $0 \leq a_i < q_i$.
Recall we abbreviate the above equation by writing $x=.a_1 a_2\cdots$ when $Q$ is understood.
Let $X_\vq$ be the set of all sequences $(a_i)_{i=1}^{\infty}$ with 
$0 \leq a_i <q_i$. 
$X_\vq$ is a compact Polish 
space with the product of the discrete topologies on the sets $\{ 0,1,\dots,q_i-1\}$.
We let $\varphi_2 \colon X_Q \to [0,1]$ be the map $\varphi_2((a_i))=.a_1a_2\dots$.
Our reduction maps will always be of the form $\varphi(x)=\varphi_2\circ \varphi_1$,
where $\varphi_1\colon \ww \to X_Q$ will vary from proof to proof.





We first prove the completeness result for distribution normality, Theorem~\ref{thm:DNQ}.

\begin{proof}[Proof of Theorem~\ref{thm:DNQ}]
Let $P=\{ x \in \ww  \colon \lim_n x(n)=\infty\}$. 
It is well-known that $P$ is $\bp^0_3$-complete. 
We define a continuous $\varphi \colon \ww  \to [0,1]$ which will be a reduction of 
$P$ to $\DNQ$, that is, such that $P= \varphi^{-1}(\DNQ)$. This suffices to show that 
$\DNQ$ is $\bp^0_3$-complete. Again, $\varphi$ will be of the form $\varphi=\varphi_2\circ \varphi_1$
where $\varphi_2$ is as above.

Let $0=b_0<b_1<b_2<\cdots$ be a sufficiently fast-growing sequence from $\N$,
so that in particular $\lim_n \frac{b_0+\cdots +b_{n-1}}{b_n}=0$. Let $I_n=[b_{n-1},b_n)$, which we call 
the $n$th block of $\N$.

Fix a $z=(z_i)_{i=1}^\infty \in X_Q$ such that $\varphi_2(z)=.z_1 z_2\cdots \in \DNQ$.
We will use $z$ as a ``reference point'' from which we make certain modifications depending
on the point $x \in \ww$.

Fix $x \in \ww$ and we describe the construction for $\varphi_1(x)$.
Let $x'(n)=\min \{ x(n),n\}$. Clearly $x \in P$ iff $x'\in P$. 
Consider the $n$th block $I_n$. 
We may assume the $b_n$ grow fast enough so that for all $n$ and all $k \leq n+2$, 
for all $b \geq b_{n-1}$ we have that 
\begin{equation}
\left| \frac{1}{b} \# \left\{ i <b \colon \tqi z \in \left[0,\frac{1}{k}\right)
\right\}  -\frac{1}{k}\right| < \frac{1}{2n}.
\end{equation}

For $i \in I_n$ we define $a_i= (\varphi_1(x))(i)$ as follows. 
For $i \in I_n$, if $\tqi z \notin [0, \frac{1}{x'(n)+2})$,
then we set $a_i=z_i$. If $\tqi z \in [0, \frac{1}{x'(n)+2})$, then 
we set $a_i$ to be the least integer in $[z_i, q_i)$ such that 
$\frac{a_i}{q_i}> \frac{1}{x'(n)+2}$. 

This defines the map $\varphi_1$, and it is clear that $\varphi_1$, and thus $\varphi$, is continuous. 
We show $\varphi$ is a reduction of $P$ to $\DNQ$.

First suppose that $x \notin P$. Then there is an $i_0$ such that for infinitely many $n$
we have $x'(n)=i_0$.  For such $n$ we have that for all $i \in I_n$ that 
$\tqi {\varphi(x)} \notin [0,\frac{1}{x'(n)+2})$, and thus 
$\tqi {\varphi(x)} \notin [0,\frac{1}{i_0+2})$. This follows from the facts that 
\[ \tqi {\varphi(x)} \geq \tqi{z},\] and if $\tqi{z}\in  [0,\frac{1}{x'(n)+2})$
then by definition of $a_i$ we have that $\tqi{\varphi(x)} \geq \frac{a_i}{q_i}
> \frac{1}{x'(n)+2}$. Since $\frac{b_0+\cdots +b_{n-1}}{b_n}$ tends to $0$,
it follows that $\varphi(x)\notin \DNQ$.

Assume now $x \in P$, and we show that $\varphi(x) \in \DNQ$. 
From Remark~\ref{rdn} it suffices to show that $\{ i \colon a_i \neq z_i\}$ has density $0$.
Fix $\epsilon >0$. Since $x \in P$, $\lim_{n \to \infty} x'(n)=\infty$. Fix $n_0$ 
large enough so that $\frac{1}{x'(n)+2} <\frac{\epsilon}{2}$ for all $n \geq n_0$.
If $n \geq n_0$, then for all $k \in I_n$ 
we have that 
\begin{equation} \label{eq:t2}
\left|\frac{1}{k} \# \left\{ i <k \colon \tqi{z} \in \left[0, \frac{1}{x'(n_0)+2}\right) 
\right\} -\frac{1}{x'(n_0)+2}  \right| < \frac{1}{2n} \leq \frac{1}{x'(n_0)+2}<\frac{\epsilon}{2}.
\end{equation}

For $n \geq n_0$, the $i \in I_n$ for which $a_i \neq z_i$ are the $i$ 
for which $\tqi{z} \in [0, \frac{1}{x'(n)+2})$, which is a subset of the $i \in I_n$
for which $\tqi{z} \in [0, \frac{1}{x'(n_0)+2})$, for $n$ large enough. From Equation~(\ref{eq:t2})
it follows that for all large enough $n$ and $k >b_n$ that 
\[
\frac{1}{k} \| \{i <k \colon a_i \neq z_i \} \|
< \epsilon.
\]
Thus,  $\{ i \colon a_i \neq z_i\}$ has density $0$.
\end{proof}

We next prove the completeness result, Theorem~\ref{NandRN}, for the classes
$\NQ$ and $\RNQ$.


\begin{proof}[Proof of Theorem~\ref{NandRN}]
Suppose $\vq$ is $1$-divergent and we first show that $\NQ$ and $\RNQ$ are $\bp^0_3$-complete.
Let $P\subseteq \ww$ be the $\bp^0_3$-complete set as before.
Fix $z =(z_i)_{i=1}^\infty \in X_Q$ such that $\varphi_2(z)=.z_1z_2,\dots \in \NQ$.



We say a block $B \in \omega^k$ is {\em good} if $\lim_{n} Q_n(B)=\infty$,
that is, the block $B$ has an infinite expectation. Recall that if  $B=(e_1,\dots,e_k)$,
then $\|B\|=\sum_{1\leq i \leq k } (e_i+1)$.

We let $0=b_0<b_1<\cdots$ be a sufficiently fast-growing sequence
so that the following hold:
\begin{enumerate}
\item \label{bp1}
$b_n-b_{n-1}> 2^n b_{n-1}$.

\item \label{bp2}
$\frac{1}{Q_m(B)} |N^Q_{m}(B,z)-Q_m(B)|<\frac{1}{n}$
for any good $B$ 
with $\|B\|\leq n$, and any $m \geq b_{n-1}$.
\item \label{bp3}
$\frac{b_{n-1}}{Q_{b_n}(B)} < \frac{1}{4^n}$ for any good $B$ with $\|B\|\leq n$.
\end{enumerate}

We define the map $\varphi_1 \colon \ww\to X_Q$, 
and our final reduction map will be $\varphi=\varphi_2\circ \varphi_1$
where $\varphi_2$ is as in the proof of Theorem~\ref{thm:DNQ}.
Let $I_n=[b_{n-1},b_n)$. For $x \in \ww$, let $x'(n)=\max \{ 27, \min \{ x(n),n\} \}$.
We define $\varphi_1(x)\res I_n$ as follows.

Let $B_0,B_1,\dots B_p$ enumerate the good blocks which occur among the
first $\lfloor \sqrt[6]{x'(n)} \rfloor$
many blocks, where we order the blocks first by $\| B\|$ and then lexicographically.
Note that this ordering of the blocks has order-type $\omega$ and the $i$th block in
this ordering has length $\leq i$. So, for $j \leq p$ we have $|B_j| \leq \sqrt[6]{x'(n)}
\leq \sqrt[6]{n}$.

For each $0\leq i \leq p$ let $m(i) \in [b_{n-1},b_n)$ be the least $m$ so that
$N_{m,b_n}(B_i,z) \leq \frac{2}{\sqrt[3]{x'(n)}} Q_{b_n}(B_i)$. Note that
$N_{m(i),b_n}(B_i,z)\geq \frac{3}{2 \sqrt[3]{x'(n)}} Q_{b_n}(B_i)$. 
Let $m=\max \{ m(i)\colon 0 \leq i \leq p\}$. Let $i_0$ be such that
$m(i_0)=m$. We note that $i_0=i_0(n)$ depends on $n$, but as $n$ is fixed
for the rest of the definition of $\varphi_1(x)\res I_n$, we will just write $i_0$.

Consider the block $B_{i_0}$. We say a block $B_j$ is {\em sparse}
relative to $B_{i_0}$ if $Q_{b_n}(B_j)< \frac{1}{4 |B_j|4^{|B_j|} \sqrt{x'(n)}} Q_{b_n}(B_{i_0})$.
Let $A\subseteq I_n$ be the set of $i \in [m,b_n]$ such that $z\res [i, i+|B_{i_0}|-1]=B_{i_0}$.
Let $r$ be the digit altering function defined as follows. If $B_{i_0}(1)=0$,
then $r$ maps  $0$ to $1$ and leaves all other values fixed. If $B_{i_0}(1)\neq 0$
then $r$ maps  $B_{i_0}(1)$ to $B_{i_0}(1)-1$ and leaves all other values fixed.
Since the block $B_{i_0}$ is fixed for the rest of the definition, the function $r$
is also. Note that $r$ applied to a valid $Q$ expansion results in a valid $Q$ expansion. 
Also, $r$ is at most $2$-to-$1$, so each block $B$ has most $2^{|B|}$ many preimages
under $r$ (we apply $r$ to a block by applying it to each digit).

Let $A' \subseteq A$ be those $i \in A$ such that for all $j<p$
such that $B_j$ is sparse relative to $B_{i_0}$, 
and all $q<|B_j|$,
$z \res[ i-q,\dots,i-q+|B_j|-1]\neq B_j$ and
$z \res[ i-q,\dots,i-q+|B_j|-1]\notin B^S_j$, where $S\subseteq |B_j|$ and
$B^S_j$ is the set of blocks $B$ such that $r(B)=B_j$ (here $r(B)$
means apply $r$ to all of the digits of $B$).

Note that if $z'$ is obtained by applying $r$ to the digits  $z_i$ for
$i$ in a subset of $A'$, then $z\res I_n$ and $z'\res I_n$ have the same number
of occurrences of the block $B_j$ for $j=1,\dots,p$.


Note that 
\[
|A'|\geq \frac{3}{2\sqrt[3] {x'(n)}}Q_{b_n}(B_{i_0}) 
-\sum_{j<p} \sum_{S\subseteq |B_j|} \sum_{r_S(B)=B_j} |B_j|  N^Q_{m,b_n}(B,z)
\]
where the first sum ranges over the $j$ such that $B_j$ is sparse relative to $B_{i_0}$,
and the second sum ranges over blocks $B$ with $r_S(B)=B_j$, where $r_S$
applies $r$ to the digits in the set $S$. 
Since $r$ either lowers
a digit or changes a $0$ to a $1$, it follows that if $r_S(B)=B_j$ then
$Q_k(B)\leq Q_k(B_j)$ for any $k$.

Thus,
\begin{equation*}
\begin{split}
|A'_n|& \geq \frac{3}{2\sqrt[3] {x'(n)}}Q_{b_n}(B_{i_0}) -
\sum_{j<p} |B_j| 4^{|B_j|} N^Q_{b_n}(B_j)
\\ &
\geq \frac{3}{2\sqrt[3] {x'(n)}}Q_{b_n}(B_{i_0}) -
(1+\frac{1}{n})\sum_{j<p} |B_j| 4^{|B_j|} Q_{b_n}(B_j)
\\ &
\geq 
\frac{3}{2\sqrt[3] {x'(n)}}Q_{b_n}(B_{i_0}) -
(1+\frac{1}{n})\sum_{j<p} |B_j| 4^{|B_j|} \left(\frac{1}{4 |B_j|4^{|B_j|} \sqrt{x'(n)}} Q_{b_n}(B_{i_0})\right)
\\ &
\geq 
\frac{3}{2\sqrt[3] {x'(n)}}Q_{b_n}(B_{i_0}) - \frac{1}{2} \frac{\sqrt[6]{x'(n)}}{\sqrt{x'(n)}} Q_{b_n}(B_{i_0})
\\ &
=
\frac{3}{2\sqrt[3] {x'(n)}}Q_{b_n}(B_{i_0}) - \frac{1}{2 \sqrt[3]{x'(n)}} Q_{b_n}(B_{i_0})
\\ &
\geq \frac{1}{\sqrt[3]{x'(n)}} Q_{b_n}(B_{i_0})
\end{split}
\end{equation*}
for all large enough $n$. Let $A''$ be the last $\frac{1}{x'(n)} Q_{b_n}(B_{i_0})$
elements of $A'$. 

For $i \in I_n$, let 
\begin{equation*}
\varphi_1(x)(i)=\begin{cases}
r(z_i) &\text{if } i \in A''\\ 
z_i &\text{otherwise}
\end{cases}
\end{equation*}
Note that $\varphi_1(x)\res I_n$ is obtained from $z \res I_n$ by changing certain digits,
the number of such changes being $|A''|=\frac{1}{x'(n)} Q_{b_n}(B_{i_0})$.
Note that the least element $i$ of $A''$ is at least the number $m$ from above
(since $A'' \subseteq A \subseteq [m,b_n]$)  and
\begin{equation*}
\begin{split}
N_m(B_{i_0},z) & \geq  N_{b_n}(B_{i_0},z)-N_{m,b_n}(B_{i_0},z
\\ &
\geq  N_{b_n}(B_{i_0},z)-\frac{2}{\sqrt[3]{x'(n)}} Q_{b_n}(B_{i_0})
\\ &
\geq \left((1-\frac{1}{n})-\frac{2}{\sqrt[3]{x'(n)}}\right) Q_{b_n}(B_{i_0}),
\end{split}
\end{equation*}
since $m=m(i_0)$ and using the definition of $m(i_0)$. 
Since $x'(n)\geq 27$, $N_m(B_{i_0},z)\geq \frac{1}{3} Q_{b_n}(B_{i_0})$ and
in particular $m \geq \frac{1}{3} Q_{b_n}(B_{i_0})$. 
From property~(\ref{bp3}) of the $b_n$  it follows that the 
first element of $A''$ is at least $\frac{4^n}{3} b_{n-1} >2^{n-1} b_{n-1}$,
for all large enough $n$.


Suppose first that $x \notin P$. There is a least $\ell$, which we call $\ell_0$,
such that for infinitely many $n$ we have that $x'(n)=\ell_0$.
At such a stage $n$ in the construction, we consider the first $\sqrt[6]{\ell_0}$
many blocks. So, for infinitely many such $n$ we may assume that the
block $B_{i_0}$ is fixed, that is, the value of $i_0$ does not depend on $n$ along this subsequence. 
Then for large enough such $n$ we have:
\begin{equation} \label{eqns1}
\begin{split}
N^Q_{b_n}(B_{i_0},\varphi_1(x)) & \leq b_{n-1}+ N^Q_{b_n}(B_{i_0},z) -\frac{1}{\ell_0} Q_{b_n}(B_{i_0})\\
&  \leq b_{n-1}+ Q_{b_n}(B_{i_0})\left(1+\frac{1}{n}\right) - \frac{1}{\ell_0} Q_{b_n}(B_{i_0})      \\
& \leq \left(\frac{1}{4^n}+1+\frac{1}{n} -\frac{1}{\ell_0}\right) Q_{b_n}(B_{i_0})\\
& \leq \left(1-\frac{1}{2\ell_0}\right) Q_{b_n}(B_{i_0})
\end{split}
\end{equation}

This shows $\varphi(x)\notin \NQ$ when $x \notin P$. 
Consider the block $B_{i_0}$ which is fixed along the subsequence. If
$B_{i_0}(1) \geq 1$, then we obtain $\varphi_1(x)\res I_n$ from $z \res I_n$ by lowering
certain occurrences of the digit $B_{i_0}(1)$ to $B_{i_0}(1)-1$. This will not
decrease the number of occurrences of the block $0_k$, where $k=|B_{i_0}|$.
So, $N^Q_{b_n}(0_k,\varphi_1(x))\geq N^Q_{b_n}(0_k,z)\geq
Q_{b_n}(0_k)(1-\epsilon)$, for small $\epsilon$ (say $\epsilon < \frac{1}{3\ell_0}$)
and all large enough $n$. This, along with Equation~\ref{eqns1},  shows $\varphi_1(x)\notin \RNQ$. 
If $B_{i_0}(1)=0$, then $r$ maps $0$ 10 $1$ and leaves all other digits unchanged.
This cannot decrease the number of occurrences of of the block $1_k$. 
So, $N^Q_{b_n}(1_k,\varphi_1(x))\geq N^Q_{b_n}(1_k,z)\geq
Q_{b_n}(1_k)(1-\epsilon)$ which again shows $\varphi_1(x)\notin \RNQ$.


Suppose now that $x \in P$ so that $\lim_n x'(n)=\infty$. 
Let $B$ be a good block, and fix $\epsilon >0$. Let $n_0>|B|$ be such that for 
all $n \geq n_0$, $x'(n)$ is large enough that $B$ is one of the first $\sqrt[6]{x'(n)}$ 
many blocks. Consider now $n \geq n_0$ and corresponding interval $I_n$.
Let $\delta>0$ and assume $n$ is sufficiently large and inductively that we have shown
$|N^Q_{b_{n-1}}(B,z)-N^Q_{b_{n-1}}(B,\varphi_1(x))| \leq \delta Q_{b_{n-1}}(B)$.

Consider first the case $B$ is sparse at stage $n$ relative to $B_{i_0}$.
Let $p$ be as before, so $p \leq \sqrt[6]{x'(n)}$.
Then $B=B_j$ for some $j \leq p$.
So, for any $i \in I_n$ we have that if $z \res[ i, i+|B|-1]=B$ or $\varphi_1(x)
\res[ i, i+|B|-1]=B$ then 
$[i,i+|B|-1]\cap A'_n=\emptyset$. Since $x \res I_n$ is obtained from $z \res I_n$ by 
changing the value only at points of $A''_n \subseteq A'_n$, it follows that for
any $k \in [b_{n-1},b_n]$ that 
$|N^Q_{k}(B,z)-N^Q_{k}(B,\varphi_1(x))|= |N^Q_{b_{n-1}}(B,z)-N^Q_{b_{n-1}}(B,\varphi_1(x))|$.
So we have for $k\in [b_{n-1},b_n]$
and large enough $n$:


\begin{equation} \label{eqnr1}
\begin{split}
|N^Q_k(B,\varphi_1(x))-Q_k(B)|& \leq |N^Q_k(B,z)-Q_k(B)|+|N^Q_{b_{n-1}}(B,\varphi_1(x))
-N^Q_{b_{n-1}}(B,z)|
\\ &
\leq |N^Q_k(B,z)-Q_k(B)|+ \delta Q_{b_{n-1}}(B)
\\ &
\leq \frac{1}{2^n}Q_k(B)+  \delta Q_{k}(B)
\\ &
\leq Q_k(B) \left(\frac{1}{2^n}+\delta \right),
\end{split}
\end{equation}
which verifies normality for the block $B$. 
Since 
$|N^Q_{b_n}(B,z)-N^Q_{b_n}(B,\varphi_1(x))|= |N^Q_{b_{n-1}}(B,z)-N^Q_{b_{n-1}}(B,\varphi_1(x))|$,
the inductive hypothesis at $b_n$ follows immediately.

Consider next the case that $B=B_j$ is not sparse at stage $n$ relative to $B_{i_0}$. 
From the definition of $i_0$ and $m$ we have that
\begin{equation} \label{eqnr2}
Q_m(B) \geq Q_{m_j}(B_j)\geq \left(1-\frac{2}{\sqrt[3]{x'(n)}}\right) Q_{b_n}(B_j).
\end{equation}
From the definition of being sparse relative to $B_{i_0}$ we have that
\begin{equation} \label{eqnr3}
Q_{b_n}(B) \geq \frac{1}{4 |B| 4^{|B|} \sqrt{x'(n)}} Q_{b_n}(B_{i_0}).
\end{equation}
Recall that $\varphi_1(x) \res I_n$ and $z\res I_n$ only differ
on $A''$, and $\min(A'')\geq m$.




Now let $k \in [b_{n-1},b_n]$, and we estimate
$|N^Q_k(B,\varphi_1(x))-Q_k(B)|$.
If $k<m$ then $|N^Q_k(B,\varphi_1(x))-Q_k(B)|\leq |N^Q_k(B,z)-Q_k(B)|+
|N^Q_{b_{n-1}}(B,z)-N^Q_{b_{n-1}}(B,\varphi_1(x))|
\leq \left( \frac{1}{2^n}+\delta \right) Q_k(B)$, which verifies normality for $B$.

So, assume $k \geq m$.
We have
\begin{equation*}
\begin{split}
|N^Q_k(B,\varphi_1(x))-Q_k(B)| & \leq b_{n-1}+ \frac{1}{x'(n)} Q_{b_n}(B_{i_0})
+ |N^Q_{k}(B,z)-Q_k(B)|
\\ &
\leq \frac{1}{2^n} Q_{b_n}(B)+  \frac{1}{x'(n)} Q_{b_n}(B_{i_0})+ \frac{1}{2^n} Q_k(B)
\\ &
\leq
\frac{1}{2^n} \frac{1}{1- \frac{2}{\sqrt[3]{x'(n)}}} Q_m(B)
+\frac{1}{x'(n)} Q_{b_n}(B_{i_0})+  \frac{1}{2^n} Q_k(B)
\\ &
\leq \frac{1}{2^{n-2}}Q_m(B)+ \frac{1}{x'(n)} Q_{b_n}(B_{i_0})+  \frac{1}{2^n} Q_k(B)
\\ &
\leq \frac{1}{2^{n-3}}Q_k(B)+ \frac{1}{x'(n)} 4|B|4^{|B|} \sqrt{x'(n)} Q_{b_n}(B)
\\ &
\leq
\frac{1}{2^{n-3}}Q_k(B)+  \frac{1}{x'(n)} 4|B|4^{|B|} \sqrt{x'(n)}
\frac{1}{1- \frac{2}{\sqrt[3]{x'(n)}}} Q_m(B)
\\ &
\leq
\frac{1}{2^{n-3}}Q_k(B)+ \frac{12|B|4^{|B|}}{\sqrt{x'(n)}} Q_k(B)
\end{split}
\end{equation*}
Since $x'(n)\to \infty$, this shows normality for the block $B$.
Similarly, letting $k=b_n$ we have
$|N^Q_{b_n}(B,\varphi_1(x))-N^Q_{b_n}(B,z)|\leq
b_{n-1}+\frac{1}{x'(n)} Q_{b_n}(B_{i_0}) \leq \frac{1}{4^n} Q_{b_n}(B)+ \frac{8|B|4^{|B|}}{\sqrt{x'(n)}} Q_{b_n}(B)
\leq \delta Q_{b_n}(B)$,
which verifies the inductive hypothesis at $b_n$.

This completes the proof that $\NQ$ and $\RNQ$ are $\bp^0_3$-complete assuming
$Q$ is $1$-divergent. If $Q$ is $1$-convergent, then every $x$ is in
$\NQ$ and $\RNQ$, so the conclusion of Theorem~\ref{NandRN} holds trivially. 
\end{proof}

The second part of Theorem~\ref{NandRN} concerning $\NkQ$ and $\RNkQ$
and the proof of  Theorem~\ref{genblocks} are slight generalizations of
the proof of Theorem~\ref{NandRN} given above. 
Since the proofs are similar we just sketch the differences.

\begin{proof}[Proof of Theorem~\ref{genblocks}]
We use the notation
and terminology of the proof of Theorem~\ref{NandRN}. 
We may assume that all of the blocks $B\in \mathscr{C}$ have infinite expectation, that is,
$\displaystyle \lim_{n \to \infty} Q_n(B)=\infty$. At stage $n$ of the construction we again let
$p=\lfloor \sqrt[6]{x'(n)} \rfloor$, and let $B_1,\dots,B_p$
enumerate the first $p$ many blocks of $\mathscr{C}$. We define the block $B_{i_0}$
as before, maximizing the value of $m(i)$ for $1 \leq i \leq p$. For
the first part of Theorem~\ref{genblocks} we may use the same digit changing function
$r$ as in the proof of Theorem~\ref{NandRN}. 
If $x\in P$, then the proof of Theorem~\ref{NandRN}
shows that $\varphi(x)\in \mathscr{N}_\mathscr{C}(Q)$. 
If $x \notin P$, then for infinitely many $n$ the value of $i_0$
will be constant, and $B_{i_0}$ is a block in $\mathscr{C}$. 
As in Equation~\ref{eqns1}, this gives an $\epsilon>0$ such that for infinitely
many $n$ we have $|N^Q_{b_n}(B_{i_0},\varphi_1(x))-Q_{b_n}(B_{i_0})|>\epsilon Q_{b_n}(B_{i_0})$.
Thus, $\varphi(x)\notin \mathscr{N}_\mathscr{C}(Q)$.
For the second statement of Theorem~\ref{genblocks} we modify the argument as the blocks
$0_k$ and $1_k$ used in the proof of Theorem~\ref{NandRN} may not be in $\mathscr{C}$.
The additional hypothesis of Theorem~\ref{genblocks},
however, guarantees the existence of a block $B_j\in \mathscr{C}$ and
an integer $t$ such that $|B_{i_0}(t)-B_j(t')|>1$ for all $t'$. As in the argument
after Equation~\ref{eqns1}, we modify the definition of $\varphi_1(x)\res I_n$
to change by $1$ all occurrences of $B_{i_0}(t)$ in $z \res I_n$
which correspond to a possible occurrence of $B_{i_0}$ (that is, the integers $i \in I_n$ where
$z\res [i-t+1,i+|B_{i_0}|]=B_{i_0}$). This will not affect the number of occurrences
of the block $B_j$ in $I_n$. This gives that $\varphi(x)\notin \mathscr{RN}_\mathscr{C}(Q)$.

\end{proof}

\section{Proof of Theorem~\ref{differences}}


We will show the $D_2(\bP^0_3)$-completeness of the non-trivial combinations of the form
$A\sm B$ where $A,B$ are one of $\NQ$, $\DNQ$, $\RNQ$.
There are $5$ non-trivial combinations, as $\NQ\subseteq \RNQ$.
\refs{4cases} handles four of these cases, which are essentially
done by the same proof. The fifth case, $\RNQ\sm \NQ$, is more complicated and
will be handled in \refs{lastcase}. We note that the
$D_2(\bP^0_3)$-completeness of the sets $A\sm B$ implies that
the  sets of the form $A\cap B$ or of the form $A\cup B$ (for $A,B \in \NQ, \DNQ, \RNQ$)
are $\bP^0_3$-complete by the following simple lemma.


\begin{lem}
If $A,B$ are $\bP^0_3$ and $A\sm B$ is $D_2(\bP^0_3)$-complete, then $A\cup B$ and $A\cap B$ are
$\bP^0_3$-complete.
\end{lem}

\begin{proof}
Suppose that $A\cup B$ were $\bS^0_3$. Then $A\sm B= (A\cup B)\sm B$ would be $\bS^0_3$, a contradiction.
Likewise, if $A\cap B$ were $\bS^0_3$ then $A\sm B=A\sm (A\cap B)$ would be $\bP^0_3$, a contradiction. 
\end{proof}

Lastly, we note that since $\NQ, \RNQ$, and $\DNQ$ are sets of full
measure by \reft{measure}, their intersections are nonempty.  We
will freely use this fact without mentioning it.

\subsection{Completeness of $\DNQ\sm \NQ$, $\DNQ\sm \RNQ$, $\NQ\sm \DNQ$, and $\RNQ\sm \DNQ$}\labs{4cases}

\begin{thm} \label{4cases}
Let $Q$ be infinite in limit and $1$-divergent.
Then the sets
$\DNQ\sm \NQ$, $\DNQ\sm \RNQ$, $\NQ\sm \DNQ$, and $\RNQ\sm \DNQ$ are all $\dt$-complete.
\end{thm}

\begin{proof}
Let $C=\{ x \in \ww \colon x(2n) \to \infty\}$, $D=\{ x \in \ww\colon x(2n+1)\to \infty\}$.
It is easy to see that $C\sm D$ is $\dt$-complete. 

Fix a fast growing sequence $\{ b_n\}$, so in particular $(2^n\sum_{i<n}b_i)/b_n \to 0$.
Recall $0_k$ denotes the sequence of length $k$ consisting of all $0$s. We introduce two basic operations
which can be performed on an interval $I\in \N^{[a,b]}$ of digits:

$\Theta_{k,\ell}$: Let $A\subseteq [a,b]$ be the set of $j$ which start an occurrence of
$0_k$, that is, $(I(j),\dots,I(j+k-1))=0_k$. Let $A'\subseteq A$ be the last $\lfloor
\frac{|A|}{\ell}  \rfloor$ many elements of $A$. For each $j \in A'$, 
change the digit $I(j)$ from a $0$ to a $1$.

$\Xi_k$: For every $j\in [a,b]$ with $I(j)\in [\frac{k}{k+1} q_j, q_j]$, change the digit
from $I(j)$ to $q_j-1$.

For the difference hierarchy results  we will use both operations,
exploiting the fact that, roughly speaking,  they allow us to modify normality/ratio normality and
distribution normality independently.

Let $z\in \DNQ \cap \RNQ$, and let $(z_i)_{i=1}^\infty\in X_Q$ be the
digits of the $Q$-Cantor series expansion of $z$.

We suppose the $b_n$ are chosen so that for all $k \leq 2n$ such that $Q$ is $k$-divergent we have
$\QNK k{b_n}> 2^n b_{n-1}$ and $\forall m \geq b_{n}\ |N^Q_m(0_k,z)-\QNK k{m}|< \frac{1}{2^n} \QNK k{m}$.

Suppose first that $\lim_n \QNK kn=\infty$ for all $k$, that is, $Q$ is fully divergent.
Given $x \in \ww$, we define $\varphi_1(x)\in X_Q$ as follows.
Let $x'(n)=\max\{ 2, \min\{ x(n),n\}\}$. 
Consider the interval of digits
$z\res I_n$, where $I_n= [b_{n-1},b_n)$. We let $\varphi_1(x)\res [b_{n-1}, b_n)$ be given by
starting with $z\res I_n$ and applying the operation $\Theta_{x'(2n),x'(2n)}$ and then the operation
$\Xi_{x'(2n+1)}$ to it. 


Recall $\varphi_2\colon X_Q\to [0,1]$ is the continuous map
\[
\varphi_2(d_1, d_2\dots)= \sum_{i=1}^\infty \frac{d_i}{q_1\cdots q_i}.
\]

We show that $\varphi=\varphi_2 \circ \varphi_1$ is a reduction from $C\sm D$ to $\NQ\sm \DNQ$. 
In fact, we show that $x \in C$ iff $\varphi(x) \in \NQ$ and $x \in D$ iff $\varphi(x)\in \DNQ$.

Since $q_i\to \infty$, the $\Theta$ operation does not affect distribution normality
as it involves changing each digit in the $Q$
Cantor series expansion by at most $1$ (see Remark~\ref{rdn}). 
Also, as $q_i\to \infty$ we have that the $\Xi$ operation does not
effect normality, since for any block of digits $B$ we have that
$|N^Q_n(B,z)-N^Q_n(B,y)|$ is bounded with $n$ (regardless of $x$),
where $y$ is the result of applying the $\Xi$ operation in all of the $I_n$.

First suppose $x\in C$, so $x(2n)\to \infty$. Let $w$ be the result of
applying just the first operation $\Theta_{x'(2n),x'(2n)}$ to $z$ in each of the intervals $I_n$.
We claim that $w \in \NQ$. 
Consider a block $B$ of digits, and let  $k=|B|$ denote its length.
First note that for any $\epsilon>0$ all large enough $n$ we have
\begin{equation*}
\begin{split}
| N^Q_{b_n}(B,w)-N^Q_{b_n}(B,z)|& \leq b_{n-1}+ \frac{2|B| }{x'(2n)} N^Q_{b_n}(0_{x'(2n)},z)
\\ &
\leq  b_{n-1}+ \frac{2|B| }{x'(2n)} \left(1+\frac{1}{2^n}\right) Q_{b_n}(0_{x'(2n)})
\\ &
\leq  b_{n-1}+ \frac{2|B| }{x'(2n)} 2 Q_{b_n}(B)
\\ &
\leq \epsilon Q_{b_n}(B).
\end{split}
\end{equation*}
Since $z \in \NQ$, for large enough $n$ we have
$| N^Q_{b_n}(B,w)- Q_{b_n}(B)|<2 \epsilon Q_{b_n}(B)$. 
Fix $\epsilon >0$, and consider now $n$ large enough and $k \in [b_{n-1},b_n)$. Let $m$ be the first element
of $A'$ in $I_n$. Note that $|A| \geq 2^n b_{n-1}$, and so $m \geq 2^{n-1}b_{n-1}$ as $x'(2n)\geq 2$.
If $k <m$ then we have
\begin{equation*}
\begin{split}
| N^Q_k(B,w)-Q_k(B)| & \leq |N^Q_{b_{n-1}}(B,w)-Q_{b_{n-1}}(B)|
+ |N^Q_{b_{n-1},k}(B,w)-Q_{b_{n-1},k}(B)|
\\ &
\leq 2\epsilon Q_{b_{n-1}}(B)+ |N^Q_{b_{n-1},k}(B,z)-Q_{b_{n-1},k}(B)|
\\ &
\leq 2\epsilon Q_{b_{n-1}}(B)+ |N^Q_k(B,z)-Q_k(B)|+ |N^Q_{b_{n-1}}(B,z)-Q_{b_{n-1}}(B)|
\\ &
\leq \left(2\epsilon +\frac{2}{2^n}\right) Q_k(B).
\end{split}
\end{equation*}

If $k\geq m$, then for large enough $n$ first note that we have
\begin{equation*}
\begin{split}
Q_k(B) & \geq Q_k(0_{x'(2n)}) \geq (1-\frac{1}{2^n}) N^Q_k(0_{x'(2n)},z)
\\ &
\geq (1-\frac{1}{2^n}) (1-\frac{1}{x'(2n)})N^Q_{b_n}(0_{x'(2n),z})
\\ &
\geq (1-\frac{1}{2^n})^2 (1-\frac{1}{x'(2n)}) Q_{b_n}(0_{x'(2n)})
\\ &
\geq \frac{1}{2} Q_{b_n}(0_{x'(2n)})
\\ &
\geq   2^{n-1} b_{n-1}.
\end{split}
\end{equation*}

So we have 
\begin{equation} \label{eqnv1}
\begin{split}
| N^Q_k(B,w)-Q_k(B)| & \leq b_{n-1}+ \frac{2|B| }{x'(2n)} N^Q_{k}(0_{x'(2n)},z)+
|N^Q_k(B,z)-Q_k(B)|
\\&
\leq \frac{2}{2^n}Q_k(B)+ \frac{2|B| }{x'(2n)} N^Q_{k}(0_{|B|},z) +\frac{1}{2^n}Q_k(B)
\\ &
\leq \frac{3}{2^n}Q_k(B)+  \frac{3 |B| }{x'(2n)} Q_{k}(B)
\\ &
\leq \epsilon Q_k(B).
\end{split}
\end{equation}
This shows that $w \in \NQ$.

If $x \notin C$, say $x(2n)= c$ for infinitely many $n$,
then for infinitely many $n$ we have
\begin{equation} \label{rr}
\begin{split}
N^Q_{b_n}(0_c,w) & \leq N^Q_{b_n}(0_c,z) +
b_{n-1}-\frac{1}{c} N^Q_{b_{n-1},b_n}(0_c,z)
\\ &
\leq \left(1+\frac{1}{2^n}\right) Q_{b_n}(0_c) +2 b_{n-1} - \frac{1}{c} N^Q_{b_n}(0_c,z)
\\ &
\leq \left(1+\frac{3}{2^n}\right) Q_{b_n}(0_c) - \frac{1}{c} \left(1-\frac{1}{2^n}\right) Q_{b_n}(0_c)
\\ &
\leq \left(1-\frac{1}{2c}\right) Q_{b_n}(0_c).
\end{split}
\end{equation}
On the other hand, the block $1_c$ occurs in $w \res [0,b_n)$
at least as many times as it does in $z\res [0,b_n)$.
So, $N^Q_{b_n}(1_c,w) \geq N^Q_{b_n}(1_c,z) \geq (1-\frac{1}{2^n}) Q_{b_n}(1_c)$.
Since $Q_{b_n}(0_c)=Q_{b_n}(1_c)$, it follows that
$w \notin \RNQ$ (and also $w \notin \NQ$).

If $x \in D$, so $x(2n+1)\to \infty$, then $\varphi(x) \in \DNQ$.
We use the fact that if $u=.u_1u_2\dots \in \DNQ$
and $v=.v_1v_2\dots$ is such that $|(u_i-v_i)/q_i|\to 0$, then $v \in \DNQ$ (see Remark~\ref{rdn}).
If $x \notin D$,
then $\varphi(x) \notin \DNQ$ since for infinitely many intervals $[b_{n-1},b_n)$
we have that $(\varphi_1(x))(i)/q_i \notin [1-\epsilon,1]$, where $\epsilon = \frac{1}{c+1}$
and $x(2n+1)=c$ for infinitely many $n$.

So we have that if $x \in C$ then $\varphi(x)\in \NQ$, and if $x \notin C$ then
$\varphi(x)\notin \RNQ$. 
Also, $x \in D$ iff $\varphi(x) \in \DNQ$. Thus, in the last two cases of the theorem,
$\varphi$ is a reduction of $C\sm D$ to the desired difference set.
For the first two cases of the theorem, $\varphi$ is a reduction of $D\sm C$ to the desired
difference set. This completes the proof of Theorem~\ref{4cases} in the case where $Q$
is fully divergent.

Assume now that that there is a largest integer $k_0$ such that $Q$ is $k_0$-divergent. 
We again obtain $\varphi_1(x)\res[b_{n-1},b_{n})$ by applying two operations.
One of these is $\Xi_{x'(2n+1)}$, where the operation $\Xi$ is as before.
For the other, we use operation $\Theta_{k_0, x'(2n)}$.

If $x \notin C$, then $\varphi(x)\notin \RNQ$ as before. If $x \in C$, that is,
$x(2n)\to \infty$, then the argument of Equation~\ref{eqnv1}
shows that $\lim_k \frac{|N^Q_k(B,w)-Q_k(B)|}{Q_k(B)}=0$ for any block
$B$ of length $\leq k_0$. Since this accounts for all of the blocks of infinite
expectation, we have that $w \in \NQ$. Since the second operation
does not affect normality, it  follows that $\varphi(x)\in \NQ$.

As before, we have that $x\in D$ iff $\varphi(x)\in \DNQ$.
So we again have that $x \in C$ implies $\varphi(x)\in \NQ$, $x \notin C$
implies $\varphi(x)\notin \RNQ$, and $x \in D$
iff $\varphi(x)\in \DNQ$. Thus, as in the previous case $\varphi$ gives
the desired reductions. 

\end{proof}

\subsection{Completeness of $\RNQ\sm \NQ$}\labs{lastcase}

We will need to define a class of functions $\ppq$ in order to prove
the last case of Theorem~\ref{differences}.
Let $P=(p_i)$ and $Q=(q_i)$ be basic sequences. If $x=a_0.a_1a_2 \cdots$ w.r.t.\ {$P$}, then put
$$
\ppq(x)=\sum_{i=1}^\infty \frac {\min(a_i,q_i-1)} {q_1 \cdots q_i}.
$$
We will need the following theorem of \cite{ppq1}.

\begin{thm}\labt{mainpsi}
Suppose that $P$ and $Q$ are basic sequences which are infinite in limit.
  If $x=a_0.a_1a_2\cdots$ w.r.t.\ $P$ satisfies
$a_i< q_{i}-1$ for infinitely many $i$, then for every block~$B$
$$
N_{i}^{Q}\left(B,\psi_{P,Q}(x)\right)
=N_{i}^{P}(B,x)+O(1).
$$
\end{thm}
While \reft{mainpsi} is not difficult to prove, it has been an
essential tool in proving some of the more difficult theorems about
$Q$-normal numbers.

Recall that for a basic sequence $Q$ and block $B$ that $Q_n(B)$ (see Equation~\ref{eqn:Qn})
denotes the expected number of occurrences of $B$ with a starting position in $[1,n]$
with respect to the basic sequence $Q$. 
Since we will be dealing with several basic sequences in this section, we extend this notation in
a natural manner. Namely, if $P$ (or $R$) are basic sequences, then we let $P_n(B)$ (or
$R_n(B)$) denote the expected number of occurrences of $B$ with a starting position in $[1,n]$
with respect to $P$ (or $R$). We similarly use the notation $N^P_n(B,z)$ 
to denote the number of occurrences of $B$ in the $P$-Cantor series for $z$ with starting position in $[1,n]$.

We will use the following lemma about concatenating intervals of normal
sequences for different basic sequences.

\begin{lem} \label{lem:t2}
Let $P$, $Q$ be basic sequences which are infinite in limit and assume that
$\limsup_i \frac{p_i}{q_i}$ is finite. 
Let $u \in \NP$, and $v \in \NQ$.
Let $B \in \omega^{<\omega}$ have infinite expectation with respect to $P$ and $Q$,
and let $\epsilon >0$. Then there is an $i_0$ such that if $i'>i\geq i_0$
then $|N^R_{i'}(B,w)-R_{i'}(B)|<\epsilon R_{i'}(B)$ where
$w\res [0,i)=u$, $w \res [i,\infty)=v$, $R\res [0,i)=P$, and
$R\res [i,\infty)=Q$.
\end{lem}


\begin{proof}
Fix $C$ such that $\frac{p_i}{q_i}\leq C$ for all $i$. 
Since $\frac{p_i}{q_i}\leq C$, and $B$ has infinite expectation with respect to $P$ and $Q$,
for large enough $i$ we have that $Q_i(B)\leq 2C^{|B|} P_i(B)$. 
Let $i_0$ be such that for all $i \geq i_0$ we have that
$|N^Q_i(B,v)-Q_i(B)|<\epsilon' Q_i(B)$ and 
$|N^P_i(B,u)-P_i(B)|<\epsilon' P_i(B)$ where $\epsilon'=\frac{\epsilon}{3+4C^{|B|}}$.
Let $i'>i \geq i_0$. We then have

\begin{equation*}
\begin{split}
|N^R_{i'}(B,w)-R_{i'}(B)|& \leq |N^P_i(B,u)-P_i(B)|+|N^Q_{i,i'}(B,v)-Q_{i,i'}(B)|
\\ &
\leq \epsilon' P_i(B)+ |N^Q_{i}(B,v)-Q_{i}(B)|+| N^Q_{i'}(B,v)-Q_{i'}(B)|
\\ &
\leq 
\epsilon' P_i(B)+ \epsilon' Q_i(B)+ \epsilon' Q_{i'}(B)
\\ &
\leq
\epsilon' R_i(B)+ 2\epsilon' Q_{i'}(B)
\\ &
\leq \epsilon' R_i(B)+ 2\epsilon' (Q_i(B)+ Q_{i,i'}(B))
\\ &
\leq \epsilon' R_i(B)+ 2\epsilon' (2C^{|B|} P_i(B) +Q_{i,i'}(B))
\\ &
\leq \epsilon' R_i(B)+ 2\epsilon' (1+2C^{|B|})R_{i'}(B)
\\ &
\leq \epsilon' (3+4C^{|B|}) R_{i'}(B) \leq \epsilon R_{i'}(B).
\end{split}
\end{equation*}

\end{proof}

We now prove the following theorem which gives the last case
of Theorem~\ref{differences}.

\begin{thm} \label{hc}
Let $Q$ be a basic sequence which is infinite in limit and
$1$-divergent.  Then the set $\RNQ\sm \NQ$ is $\dt$-complete.
\end{thm}

\begin{proof}
Let $C, D\subseteq \ww$ be as in \refs{4cases}.
For $k, n\in \N$, recall that $\QNK kn$ denotes the sum
$\QNK kn= \sum_{i=1}^n \frac{1}{q_i q_{i+1}\cdots q_{i+k-1}}$.


For each $m$, let $P_m=( \max(2,\lfloor \frac{m+1}{m+2} q_i \rfloor))_{i=1}^\infty$.
Recall $(P_m)_i(B)$ denotes the  expectation of $B$ in the first $i$
digits of $P_m$, as in Equation~\ref{eqn:Qn},
and $(P_m)_i^{(k)}$ denote the expectation of $0_k$ in the first $i$ digits of $P_m$.
We fix for the rest of the proof
$w_m\in \mathscr{N}(P_m)$, which we identify with a $P_m$-Cantor series expansion.

Given a strictly increasing sequence $\{ b_n\}$ (which we will choose below)
and an $x \in \ww$ we define a new basic sequence $P_x=P(x, \{ b_n\},Q)=(p^x_i)$ as follows.
Let $x'(n)=\min \{ x(n),n\}$.
Let $p^x_i=P_{x'(2n+1)}(i)=\max (2,\lfloor \frac{x'(2n+1)+1}{x'(2n+1)+2} q_i \rfloor) $
for $i \in [b_{n-1},b_{n})$.
Note that $p_i \leq q_i$ for all $i$,
and if $x(2n+1)\to \infty$ then $\lim_i \frac{p^x_i}{q_i}=1$. Also
$\QNK kn \leq \PxNK kn$ for all $k, n$. 

\begin{claim} \label{tc}
Suppose $Q$ is $k$-divergent.
If $x(2n+1)\to \infty$ then $\lim_i \frac{\QNK ki}{\PxNK ki} \to 1$.  If $x(2n+1)$
does not tend to $\infty$ and $\lim_n \frac{\QNK k{b_{n-1}}}{\QNK k{b_{n-1},b_{n}}}=0$,
there is a subsequence of $\frac{\QNK ki}{\PxNK ki}$ which is bounded away from $1$.
\end{claim}

\begin{proof}
If $x(2n+1)\to \infty$, then $\ell_n=\frac{x(2n+1)+1}{x(2n+1)+2} \to 1$
and so $p^x_i/q_i$ tends to $1$. Note that
$\left( \frac{1}{q_jq_{j+1}\cdots q_{j+k-1}}\right) / \left( \frac{1}{p_jp_{j+1}\cdots p_{j+k-1}}\right)
=(\ell_n)^k$ for all $j \in [b_n,b_{n+1})$. We then have that
for all $k$ that $\QNK ki/ \PxNK ki\to 1$ using the simple fact that if $c_j,d_j \geq 0$,
$\sum_j c_j=\infty$, $\sum_j d_j=\infty$, and $c_j/d_j \to 1$, then
$\left( \sum_{j=1}^i c_j \right) / \left( \sum_{j=1}^i d_j \right) \to 1$.

If $x(2n+1)$ does not tend to $\infty$, then $\exists u <1$ such that $\ell_n \leq u$
for infinitely many $n$. So, for infinitely many $n$ we have that
$\QNK k{b_{n-1}, b_{n}} \leq u^k \PxNK k{b_{n-1}, b_{n}}$.
Thus, for infinitely many $n$ 
\begin{equation*}
\begin{split}
\frac{\QNK k{b_{n}}}{\PxNK k{b_{n}}}& = \frac{\QNK k{b_{n-1}}+ \QNK k{b_{n-1}, b_{n}}}
{\PxNK k{b_{n-1}}+ \PxNK k{b_{n-1}, b_{n}}}
\leq \frac{\QNK k{b_{n-1}}+ \QNK k{b_{n-1}, b_{n}}} {\QNK k{b_{n-1}}+ \PxNK k{b_{n-1}, b_{n}}}
\leq \frac{\QNK k{b_{n-1}}+ \QNK k{b_{n-1}, b_{n}}} {\QNK k{b_{n-1}}+ (\frac{1}{u^k}) \QNK k{b_{n-1}, b_{n}}}
\\ &
\leq \frac{\QNK k{b_{n-1}}+ (\frac{1}{u^k}) \QNK k{b_{n-1}, b_{n}}+(1-\frac{1}{u^k}) \QNK k{b_{n-1}, b_{n}}}
{\QNK k{b_{n-1}}+ (\frac{1}{u^k}) \QNK k{b_{n-1}, b_{n}}}
\\ &
= 1- \left( \frac{1}{u^k} -1\right) \frac{ \QNK k{b_{n-1}, b_{n}}}{\QNK k{b_{n-1}}+
(\frac{1}{u^k}) \QNK k{b_{n-1}, b_{n}}}
\\ &
\leq 1- \left( \frac{1}{u^k}-1 \right) \frac{ \QNK k{b_{n-1}, b_{n}}}{ (\frac{2}{u^k}) \QNK k{b_{n-1}, b_{n}}}
= 1- \left( \frac{1}{u^k} -1\right) \frac{u^k}{2}
\\ &
= 1- \frac{1}{2} (1-u^k)
\end{split}
\end{equation*}

\end{proof}

First assume that $Q$ is $k$-divergent for all $k$. 
Using Lemma~\ref{lem:t2} we then inductively pick the $b_n$ satisfying the following:

\begin{enumerate}
\item \label{lp1}
For all $m_1,m_2 \leq n+1$ and all $i>b_{n-1}$, let $P=P(m_1,m_2,i)$
be defined by: $P\res [0,b_{n-1})= P_{m_1} \res [0,b_{n-1})$ and $P\res [b_{n-1},i)= P_{m_2}
\res [b_{n-1},i)$. Let $w\res [0,b_{n-1})=w_{m_1}$ and $w \res [b_{n-1},i)=w_{m_2}\res [b_{n-1},i)$.
Then for any $B$ with $\|B\|\leq n+1$
we have $|N^P_i(B, w)-P_i(B)|<\frac{1}{2^n} P_i(B)$.
\item \label{lp2}
$(P_m)_{b_{n-1},b_n}(0_{2n})  > 2^n b_{n-1}$ for all $m \leq n$.
\end{enumerate}

Given $x \in \ww$, we define $\varphi_1(x)$ as follows. Suppose $\varphi_1(x)\res b_{n-1}$ has been defined.
Let $y \res [b_{n-1},b_{n})= w_{x(2n+1)}\res [b_{n-1},b_{n})$. 
Then we perform the operation $\Theta_{x'(2n),x'(2n)}$ of Section~\ref{section:4cases}
on $y \res [b_{n-1},b_{n})$ to produce
$\varphi_1(x) \res [b_{n-1},b_{n})$. 
This defines $\varphi_1(x)\res [b_{n-1},b_{n})$. Doing this for all blocks $[b_{n-1},b_{n})$ produces $\varphi_1(x)$.

If $x \notin C$, that is $x(2n)$ does not tend to $\infty$, then $\varphi(x) \notin \RNQ$.
This is because if $k=\liminf x(2n)$, then there will be infinitely many $n$ for which 
$0_k$  occurs in $[0,b_{n})$ at most
\begin{equation}
\begin{split}
N^{P_{x'(2n+1)}}_{b_n}(0_k,w_{x'(2n+1)})(1-\frac{1}{k})+b_{n-1} &
\leq (P_{x'(2n+1)})^{(k)}_{b_{n}} (1-\frac{1}{k})(1+\frac{1}{2^n})+b_{n-1}
\\ &
\leq (P_{x'(2n+1)})^{(k)}_{b_{n}} (1-\frac{1}{k}+\frac{1}{2^{n-1}})
\\ &
\leq  (P_{x'(2n+1)})^{(k)}_{b_{n}} (1-\frac{1}{2k})
\end{split}
\end{equation}
many times 
while $1_k$ occurs  at least
\begin{equation}
\begin{split}
N^{P_{x'(2n+1)}}_{b_{n}}(1_k,w_{x'(2n+1)}) -b_{n-1}& \geq (P_{x'(2n+1)})^{(k)}_{b_{n}} (1-\frac{1}{2^n})-b_{n-1}
\\ & \geq (P_{x'(2n+1)})^{(k)}_{b_{n}} (1-\frac{1}{2^{n-1}})
\end{split}
\end{equation}
many times.

If $x(2n)\to \infty$ but $x(2n+1)$ does not tend to infinity, then $\varphi(x)\in \RNQ\sm \NQ$.
To see this, first note that the point $y$ as above is in $\RNQ$.
Recall $P_x$ is defined by $P_x \res[b_{n-1},b_{n})=P_{x(2n+1)}\res [b_{n-1},b_{n})$.
We show that $y \in \mathscr{N}(P_x)$, which implies $y \in \RNQ$. 
For any $B$ and for large enough $n$ and for any 
$b_{n-1} \leq i <b_{n}$ we have from property (\ref{lp1}) of the $b_n$:
\begin{equation} \label{qq1}
| N^{P'}_i(B,y')-P'_i(B)|< \frac{1}{2^n} P'_i(B),
\end{equation}
where $y'$ and $P'$ are defined by:
\begin{equation*}
\begin{split}
y'\res[0,b_{n-1})&=w_{x(2n-1)} \res [0,b_{n-1})
\\ 
y' \res [b_{n-1},i)&=w_{x(2n+1)}\res i
\\ 
P' \res [0,b_{n-1})&=P_{x(2n-1)}
\\ 
P'\res [b_{n-1},i)&= P_{x(2n+1)}\res [b_{n-1},i).
\end{split}
\end{equation*}

Also, from property (\ref{lp2}) of the $b_n$ we have:
\begin{equation} \label{qq2}
|N^{P'}_i(B,y')-N^{P_x}_i(B,y)|\leq b_{n-2}<\frac{1}{2^n}
(P_x)_{b_{n-1}}(B)\leq \frac{1}{2^n} (P_x)_i(B).
\end{equation}
Finally,
\begin{equation} \label{qq3}
|P'_i(B)-(P_x)_i(B)|\leq b_{n-2}< \frac{1}{2^n} (P_x)_i(B).
\end{equation}
From Equations \ref{qq1}, \ref{qq2} and \ref{qq3}
we have $|N^{P_x}_i(B,y)-(P_x)_i(B)|\leq \frac{4}{2^{n}} (P_x)_i(B)$.
This shows that $y \in \RNQ$.
Since $x(2n+1)$ does not tend to infinity, then from Claim~\ref{tc} there is a
subsequence  
on which $\frac{Q^{(k)}_{i}}{(P_x)^{(k)}_{i}}$ is bounded away from $1$.
Since $y$ is $P_x$-normal,
we have that $y$ is not $Q$-normal. 

The operation
applied to $y$ to produce $\varphi_1(x)$ does not affect normality or ratio normality
if $x(2n) \to \infty$ (this is just as in ~\refs{4cases}).
So, $\varphi(x)  \in \RNQ \sm \NQ$.

Finally, if $x(n)\to \infty$, then as above $y \in \mathscr{N}(P_x)$.
As $x(2n+1)\to \infty$, we have from Claim~\ref{tc}
that $\lim_i \frac{\QNK ki}{\PxNK ki} \to 1$ and it follows that 
$y \in \NQ$.
Since $x(2n)\to \infty$ as well, from the argument in \refs{4cases}
we also have that $y=\varphi(x)\in \NQ$.

So, in all cases we have that $x \in C\sm D$ iff $\varphi(x) \in \RNQ\sm \NQ$.

Suppose now that there is a largest integer $k_0$ such that $Q$ is $k_0$-divergent.
The proof is essentially identical to that above. We let $P_m$ be as before,
and now let $w_m$ be $\leq k_0$ normal with respect to $P_m$, that is, for all $B$
of length $\leq k_0$
we have $\lim_i \frac{N_i(B,w_m)}{(P_m)_i(B)}=1$. We define $\varphi_1(x)$ by first
defining $y$ exactly as before (using the values $x(2n+1)$).
We then modify $y$ to $\varphi_1(x)$, (using $x(2n)$) but in a slightly different manner.
Namely, we get $\varphi_1(x)\res [b_{n-1},b_{n})$ from $y \res [b_{n-1},b_{n})$ as follows.
Let $A\subseteq [b_{n-1},b_n)$ be the integers $i$ in this interval such that
$y \res [i,i+k_0-1)= 0_{k_0}$. Let $A' \subseteq A$ be the last $\lfloor \frac{|A|}{x(2n)}\rfloor$
many elements of $A$. For each $i \in A'$ we change $y(i)$ from a $0$ to a $1$,
and for all other $i$ in this interval we set $\varphi_1(x)(i)=y(i)$.

If $x(2n)$ does not tend to infinity,
then easily $\varphi(x) \notin \RNQ$ as
$\frac{N^{P_x}_i(1_{k_0},\varphi_1(x))}{N^{P_x}_i(0_{k_0},\varphi_1(x))}$
does not tend to $1$. If $x(2n)$ tends to infinity, then we easily have that $\varphi(x)$ is
in $\NQ$ (or $\RNQ$) iff $y$ is in $\NQ$ (resp.\ $\RNQ$). In this case, as above, we have that if
$x(2n+1)\to \infty$ then $y\in \NQ$, and if $x(2n+1)$ does not tend to infinity
then $y \in \RNQ\sm \NQ$. So, in all cases we have $x \in C\sm D$ iff
$\varphi(x)\in \RNQ\sm \NQ$. 

\end{proof}

\subsection{Further Discussion}

Theorem~\ref{4cases} can be extended further. First, the hypothesis that
$Q=(q_i)$ in infinite in limit can be weakened to the following condition
studied by T.\ \v{S}al\'{a}t \cite{Salat}: $\lim_{N\to \infty} \frac{1}{N} \sum_{i=1}^N \frac{1}{q_i}=0$.
This condition is equivalent to saying that there is a set $D\subseteq \N$ of
density $0$ such that $(q_i)_{i \notin D}$ tends to infinity (see Theorem~1.20 of \cite{Walters}).
Since changing a sequence on a set of density $0$ may affect normality and ratio normality,
we must now use the argument of Theorem~\ref{NandRN}. At stage $n$ of the construction of
$\varphi_1(x)\res I_n$, we again use two operations $\Theta'_{x'(2n)}$ and $\Xi'_{x'(2n+1)}$.
The first operation $\Theta'_{x'(2n)}$ is the operation implicitly described in the proof of Theorem~\ref{NandRN}.
That is, we define the block $B_{i_0}$ exactly as in that proof, and define the sets
$A, A',A''\subseteq I_n$ as in that proof. We then eliminate the occurrences of the block $B_{i_0}$
at the points of $A''$ by applying the digit changing function $r$ as in Theorem~\ref{NandRN}.
Let $w$ be the result of applying this first operation to $z$ (so $w$ is the
$\varphi_1(x)$ of Theorem~\ref{NandRN}). The proof of Theorem~\ref{NandRN}
did not require that $Q$ be infinite in limit, and so we have that
$x(2n)\to \infty$ implies $w \in \NQ$ and $x(2n) \nrightarrow \infty$ implies
$w \notin \RNQ$. The function $r$ changes digits by at most $1$, and does not affect
distribution normality using Remark~\ref{rdn} and the fact that $D$ has density $0$
(changing a sequence on a set of density $0$ does not affect distribution normality). 
So, $w \in \DNQ$. The second operation $\Xi'_{x'(2n+1)}$ is the operation $\Xi_{x'(2n+1)}$
of Theorem~\ref{4cases} except we only apply the operation to digits not in $D$. 
We let $\varphi_1(x)$ be the result of applying these operations to $w$.
The operations $\Xi'_{x'(2n+1)}$ do not affect normality or
distribution normality as $q_i\to \infty$ off of $D$, and so for every block $B$,
$|N^Q_m(B,\varphi_1(x))-N^Q_m(B,w)|$ is bounded with $m$. As in Theorem~\ref{4cases}
we have that $\varphi_1(x)\in \DNQ$ iff $x(2n+1)\to \infty$. So, $\varphi=\varphi_2\circ
\varphi_1$ is a reduction of $C\sm D$ (or $D\sm C$ depending on the case) to
the desired set.

Second, we can prove the version of Theorem~\ref{4cases} with $\NQ$ and $\RNQ$
replaced with $\NkQ$ and $\RNkQ$, provided we assume that $Q$ is $k$-divergent
(and infinite in limit, or more generally
$\lim_{N\to \infty} \frac{1}{N} \sum_{i=1}^N \frac{1}{q_i}=0$).
We proceed as above except in defining the block $B_{i_0}$ used in the first
operation, we only consider the first $\sqrt[6]{x'(2n)}$ many good blocks
$B_1,\dots,B_p$ of length $k$. This makes sense since there is some
block of length $k$, namely $0_k$, which has infinite expectation.
If $w$ again denotes the result of applying the first operation in
all of the $I_n$, then the proof of Theorem~\ref{NandRN} shows that
if $x(2n)\to \infty$ then $|N^Q_m(B,w)-N^Q_m(B,z)|/Q_m(B)\to 0$ for all
blocks $B$ of length $k$. It follows that if $x(2n)\to \infty$
then $w \in \NkQ$ and if $x(2n) \nrightarrow \infty$ then
$w \notin \RNkQ$. Also, $w \in \DNQ$ as above. The second operation
works exactly as in the above argument, so $\varphi_1(x) \in \DNQ$
iff $x(2n+1)\to \infty$. So, $\varphi=\varphi_2\circ \varphi_1$
again gives the desired reduction.

\bibliographystyle{amsplain}


\input{genbases_revised_2.bbl}

\end{document}

%% file: genbases_revised_2.bbl
\providecommand{\bysame}{\leavevmode\hbox to3em{\hrulefill}\thinspace}
\providecommand{\MR}{\relax\ifhmode\unskip\space\fi MR }
\providecommand{\MRhref}[2]{%
  \href{http://www.ams.org/mathscinet-getitem?mr=#1}{#2}
}
\providecommand{\href}[2]{#2}

%% file: genbases_revised_2.bbl
\begin{thebibliography}{10}

\bibitem{AJKMSubshifts}
D.~Airey, S.~Jackson, D.~Kwietniak, and B.~Mance, \emph{Borel complexity of
  sets of normal numbers via generic points in subshifts with specification},
  To appear, Transactions of the American Mathematical Society.

\bibitem{AJKMDynamics}
\bysame, \emph{Borel complexity of the set of generic points of dynamical
  systems with a specification property}, in preparation.

\bibitem{AireyManceHDDifference}
D.~Airey and B.~Mance, \emph{On the {H}ausdorff dimension of some sets of
  numbers defined through the digits of their {$Q$}-{C}antor series
  expansions}, J. Fractal Geom. \textbf{3} (2016), no.~2, 163--186, MR3501345.

\bibitem{AireyManceVandehey}
D.~Airey, B.~Mance, and J.~Vandehey, \emph{Normality preserving operations for
  {C}antor series expansions and associated fractals, {II}}, New York J. Math.
  \textbf{21} (2015), 1311--1326, MR3441645.

\bibitem{BecherHeiberSlamanAbsNormal}
V.~Becher, P.~A. Heiber, and T.~A. Slaman, \emph{Normal numbers and the {B}orel
  hierarchy}, Fund. Math. \textbf{226} (2014), no.~1, 63--78.

\bibitem{BecherSlamanNormal}
V.~Becher and T.~A. Slaman, \emph{On the normality of numbers to different
  bases}, J. Lond. Math. Soc. (2) \textbf{90} (2014), no.~2, 472--494.

\bibitem{Cantor}
G.~Cantor, \emph{\"{U}ber die einfachen {Z}ahlensysteme}, Zeitschrift f\"{u}r
  Math. und Physik \textbf{14} (1869), 121--128.

\bibitem{ErgNum}
K.~Dajani and C.~Kraaikamp, \emph{Ergodic theory of numbers}, Carus
  Mathematical Monographs, vol.~29, Mathematical Association of America,
  Washington, DC, 2002.

\bibitem{ErdosRenyiConvergent}
P.~Erd\H{o}s and A.~R\'{e}nyi, \emph{On {C}antor's series with convergent $\sum
  1/q_n$}, Annales Universitatis L. E\"{o}tv\"{o}s de Budapest, Sect. Math.
  \textbf{2} (1959), 93--109.

\bibitem{ErdosRenyiFurther}
\bysame, \emph{Some further statistical properties of the digits in {C}antor's
  series}, Acta Math. Acad. Sci. Hungar \textbf{10} (1959), 21--29.

\bibitem{Galambos}
J.~Galambos, \emph{Representations of real numbers by infinite series}, Lecture
  Notes in Math., vol. 502, Springer-Verlag, Berlin, Hiedelberg, New York,
  1976.

\bibitem{Kechris}
A.~Kechris, \emph{Classical descriptive set theory}, Graduate Texts in
  Mathematics, vol. 156, Springer-Verlag, New York, 1995.

\bibitem{KiLinton}
H.~Ki and T.~Linton, \emph{Normal numbers and subsets of {N} with given
  densities}, Fund. Math. \textbf{144} (1994), no.~2, 163--179.

\bibitem{Mance}
B.~Mance, \emph{Construction of normal numbers with respect to the {$Q$}-cantor
  series expansion for certain {$Q$}}, Acta Arith. \textbf{148} (2011),
  135--152.

\bibitem{Mance4}
\bysame, \emph{Typicality of normal numbers with respect to the {C}antor series
  expansion}, New York J. Math. \textbf{17} (2011), 601--617.

\bibitem{ppq1}
\bysame, \emph{Number theoretic applications of a class of {C}antor series
  fractal functions, {I}}, Acta Math. Hungar. \textbf{144} (2014), no.~2,
  449--493.

\bibitem{RauzyNormalPreserving}
G.~Rauzy, \emph{Nombres normaux et processus d\'{e}terministes}, Acta Arith.
  \textbf{29} (1976), no.~3, 211--225.

\bibitem{RenyiProbability}
A.~R\'{e}nyi, \emph{On a new axiomatic theory of probability}, Acta Math. Acad.
  Sci. Hungar. \textbf{6} (1955), 329--332.

\bibitem{Renyi}
\bysame, \emph{On the distribution of the digits in {C}antor's series}, Mat.
  Lapok \textbf{7} (1956), 77--100.

\bibitem{RenyiSurvey}
\bysame, \emph{Probabilistic methods in number theory}, Shuxue Jinzhan
  \textbf{4} (1958), 465--510.

\bibitem{Turan}
P.~Tur\'{a}n, \emph{On the distribution of ``digits'' in {C}antor systems},
  Mat. Lapok \textbf{7} (1956), 71--76.

\bibitem{Salat}
T.~\u{S}al\'{a}t, \emph{Zu einigen {F}ragen der {G}leichverteilung (mod 1)},
  Czech. Math. J. \textbf{18 (93)} (1968), 476--488.

\bibitem{Wall}
D.~D. Wall, \emph{Normal numbers}, Ph.D. thesis, Univ. of California, Berkeley,
  Berkeley, California, 1949.

\bibitem{Walters}
P.~Walters, \emph{An introduction to ergodic theory}, Springer-Verlag, New
  York, 1982.

\end{thebibliography}
